\documentclass[twoside, 11pt, leqno]{amsart}
 \usepackage{amsmath, amssymb, amsfonts, amsthm, amscd} \usepackage{mathrsfs}
  \usepackage[pdftex]{graphicx}

\long\def\symbolfootnote[#1]#2{\begingroup
\def\thefootnote{\fnsymbol{footnote}}\footnote[#1]{#2}\endgroup}

\theoremstyle{plain}
\numberwithin{equation}{section}

\newtheorem{theorem}{Theorem}[section]
\newtheorem{lemma}[theorem]{Lemma}
\newtheorem{corollary}[theorem]{Corollary}
\newtheorem{proposition}[theorem]{Proposition}
\newtheorem{definition}[theorem]{Definition}
\newtheorem{remark}[theorem]{Remark}

\newcommand{\fracsm}[2]{\begin{matrix}\frac{#1}{#2}\end{matrix}}

\newcommand{\beq}{\begin{equation}}
\newcommand{\eeq}{\end{equation}}
\newcommand{\ep}{\varepsilon}             % General stuff

\newcommand{\Reals}{\mathbb{R}}

\newcommand{\Id}{\mathrm{Id}}

\DeclareMathOperator{\Ric}{Ric}
\DeclareMathOperator{\Hess}{Hess}
\DeclareMathOperator{\dRham}{d}
\DeclareMathOperator{\trace}{tr}
\DeclareMathOperator{\diverg}{div}
\DeclareMathOperator{\grad}{\nabla}
\DeclareMathOperator{\Proj}{Proj_{\{x\cdot y=0\}}}

\begin{document}
\title{Mean curvature self-shrinkers of high genus: Non-compact examples}
\thanks{The results and methods in this paper have been presented at conferences in Princeton (February 2011), and at seminars at the Max Planck Institute in Golm (March 2011) and at MIT (May 2011). The first and second authors were supported partially by NSF grants DMS1105371 and  DMS1004646 , respectively.}

\author{N. Kapouleas}
\address{Nikolaos Kapouleas, Department of Mathematics, Brown University, Providence, RI 02912.}
\email{nicos@math.brown.edu}

\author{S. J. Kleene}
\address{Stephen James Kleene, MIT, Cambridge, MA 02139.}
\email{skleene@math.mit.edu}

\author{N. M. M\o{}ller}
\address{Niels Martin M\o{}ller, MIT, Cambridge, MA 02139.}
\email{moller@math.mit.edu}

\begin{abstract}
We give the first rigorous construction of complete, embedded self-shrinking hypersurfaces under mean curvature flow, since Angenent's torus in 1989. The surfaces  exist for any sufficiently large prescribed genus $g$, and are non-compact with one end. Each has $4g+4$ symmetries and comes from desingularizing the intersection of the plane and sphere through a great circle, a configuration with very high symmetry.

Each is at infinity asymptotic to the cone in $\Reals^3$ over a $2\pi/(g+1)$-periodic graph on an equator of the unit sphere $\mathbb{S}^2\subseteq \Reals^3$, with the shape of a periodically "wobbling sheet". This is a dramatic instability phenomenon, with changes of asymptotics that break much more symmetry than seen in minimal surface constructions.

The core of the proof is a detailed understanding of the linearized problem in a setting with severely unbounded geometry, leading to special PDEs of Ornstein-Uhlenbeck type with fast growth on coefficients of the gradient terms. This involves identifying new, adequate weighted H\"older spaces of asymptotically conical functions in which the operators invert, via a Liouville-type result with precise asymptotics.
\end{abstract}

%\today

\maketitle

\section{Introduction}
In studying the flow of a hypersurface by mean curvature in Euclidean $n$-space as well as in general ambient Riemannian $n$-manifolds $(M^n,g)$, $n\geq 3$, the basic ``atoms" of singularity theory are the self-similar surfaces in $\Reals^n$, viz. solitons moving by an ambient conformal Killing field, and of these the self-shrinkers are the most important. Taking center stage when identified by Huisken in 1988 (and the compact $H\geq0$ case classified: Round spheres; see \cite{Hu90}) as the surfaces for which equality holds in his celebrated monotonicity formula, the self-shrinkers arise as blow-up limits when assuming natural curvature bounds.

It is notable that even when $n=3$ only a few complete, embedded self-shrinking surfaces in $\Reals^3$ are to this date rigorously known: Flat planes, round cylinders, round spheres and a (not round-profile) torus of revolution discovered by Angenent in \cite{An} (this list exhausts the rotationally symmetric examples, although the uniqueness of the torus is still open; see \cite{KM}). Note also that several results involving self-shrinkers in some generality have appeared, most prominently a smooth compactness theorem (for closed, fixed genus surfaces \cite{CM1}) and a theory of generic singularities of Colding-Minicozzi, including classification of all $H\geq 0$ complete hypersurfaces (see \cite{CM2} and \cite{DX}). See also \cite{LS} and \cite{Wa} for other uniqueness results.

The self-shrinker equation is a nonlinear partial differential equation of mean curvature type, indeed the self-shrinkers are minimal with respect to a certain Gaussian metric on Euclidean space, and as such the current status of known examples can be likened to the situation before Scherk's, Riemann's and Enneper's minimal surface examples, when only rotationally symmetric surfaces were known. In recent years, several authors (\cite{Tr96}, \cite{Ka97}, \cite{Ka05}, \cite{Ka11}) have, via singular perturbation techniques, greatly expanded upon the list of rigorously known minimal surfaces in $\Reals^3$. Since the local considerations involved in the constructions would work in some generality (see \cite{Ka05} and \cite{Ka11}), it has long been expected that such constructions could work for self-similar surfaces under mean curvature flow, and indeed there are constructions for the self-translating case in the interesting work by X.H. Nguyen (see \cite{Ng1}-\cite{Ng2}).

The existence of self-shrinkers with the topology we consider in this paper was conjectured by Tom Ilmanen in 1995 (from numerics, using Brakke's surface evolver; see \cite{Il95}), while their asymptotic geometry was not clear at that point.

Our main theorem is the following:
\begin{theorem}\label{Thm_main}
For every large enough integer $g$ there exists a complete, embedded, orientable, smooth surface $\Sigma_g\subseteq \Reals^3$, with the properties:
\begin{itemize}
\item[(i)] $\Sigma_g$ is a mean curvature self-shrinker of genus $g$.
\item[(ii)] $\Sigma_g$ is invariant under the dihedral symmetry group with $4g+4$ elements.
\item[(iii)] $\Sigma_g$ has one non-compact end, and separates $\Reals^3$ into two connected components.
\item[(iv)] The end is outside some Euclidean ball a graph over a plane, asymptotic to the cone on a non-zero vertical smooth (4g+4)-symmetric graph over a great circle in $\mathbb{S}^2$ (hence the visual appearance of a "wobbling sheet").
\item[(v)] Inside any fixed ambient ball $B_{R}(0)\subseteq\Reals^3$, the sequence $\{\Sigma_g\}$ converges in Hausdorff sense to the union $\mathbb{S}^2\cup\mathcal{P}$, where $\mathcal{P}$ is a plane through the origin in $\Reals^3$. In fact, the bounds
\beq\label{Hausdorff}
d_H\big[\Sigma_g\cap B_{R}(0),(\mathbb{S}^2\cup \mathcal{P})\cap B_{R}(0)\big]\leq C \frac{R}{g},
\eeq
on the Hausdorff distance $d_H$ hold for some constant $C>0$. The convergence is furthermore locally smooth away from the intersection circle.
\end{itemize}
\end{theorem}

\begin{corollary}
Euclidean flat cylinders over $\Sigma_g$ are shrinkers. So, in any fixed dimension $n\geq2$ we obtain self-shrinking hypersurfaces $\Sigma_g^{n}=\Sigma_g\times\Reals^{n-2}\subseteq\Reals^{n+1}$, with arbitrary large first Betti number
\[
b_1(\Sigma_g\times\Reals^{n-2})=b_1(\Sigma_g)=2g.
\]
\end{corollary}
\begin{figure}
\includegraphics[height = 300pt]{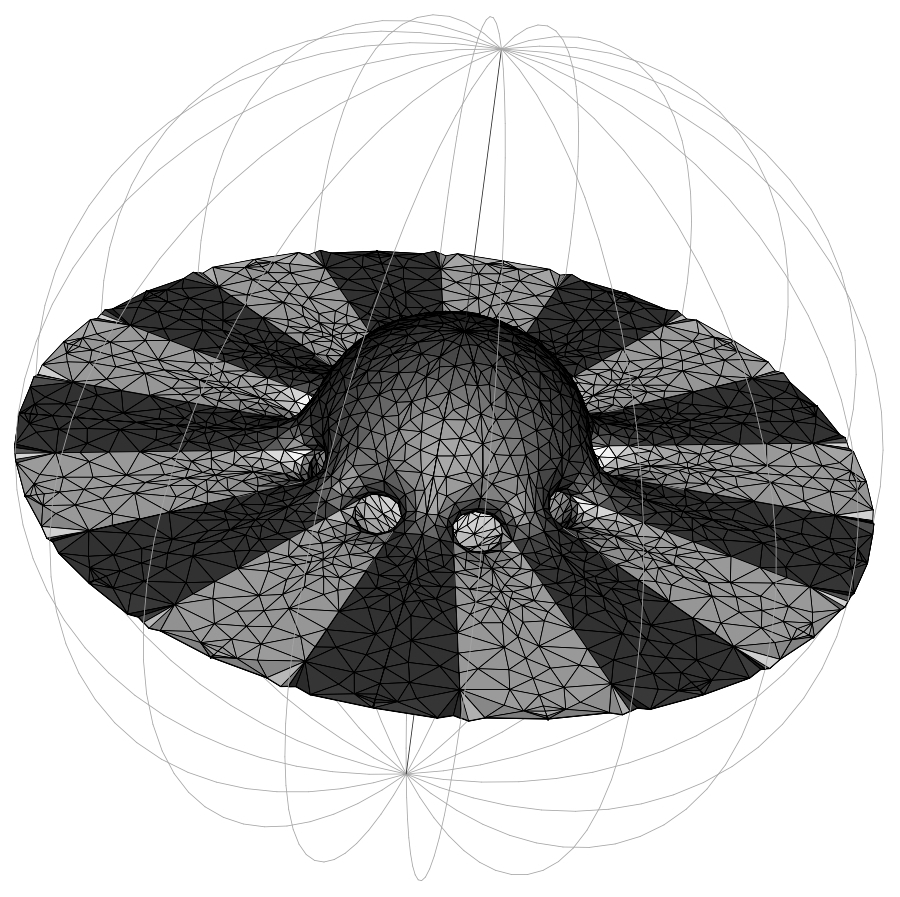}
\caption{Tom Ilmanen's conjectural shrinker of genus 8 with 9 Scherk handles.}
\end{figure}
The general approach of this article is the same as that of \cite{Ka97}, which follows the general methodology developed in  \cite{Ka95}. Our construction is analogous to a specific instance of the main theorem in \cite{Ka97},  the case of a catenoid  intersecting a plane through its waist, which is simpler than the general case because of the extra symmetry. On the other hand, we must contend with major analytic difficulties arising from the unbounded nature of the self-shrinker equation, which do not arise in minimal and constant mean curvature constructions.

To look further into the analytical difficulties faced here, it is instructive to use the mentioned characterization of self-shrinkers (which shrink towards the origin, with scaling factor $\sqrt{2(1-t)}$): Minimal surfaces $S\subseteq\Reals^3$ w.r.t. the conformal metric $g_{ij}=e^{-|x|^2/4}\delta_{ij}$, where $|x|$ is the distance to the origin and $\delta_{ij}$ is the Euclidean standard metric. All previous desingularization constructions -- and indeed much of geometric analysis -- rely on some kind of reasonably bounded geometry such as for example geodesic completeness, curvature bounds, or even stronger assumptions such as asymptotic flatness. We must however here face that the metric is geodesically incomplete (non-extendible: the distance to infinity is finite) and the Ricci curvature of a plane through the origin in the unit normal direction, respectively the Gauss curvature of the induced metric on such a plane, are (see Appendix C):
\begin{align*}
&\Ric(\vec{\nu},\vec{\nu})=e^{|x|^2/4}(1-|x|^2/16)\to-\infty,\quad\textrm{for}\quad |x|\to\infty,\\
&K_{(\Reals^2,g)}=\frac{1}{2}e^{|x|^2/4}\to+\infty,\quad\textrm{for}\quad |x|\to\infty.
\end{align*}

It should hence come as no surprise that the analysis we need to perform could not follow from any very general principle, and in fact this paper also gives the first successful example of a construction for such an unbounded geometry. Our new (anisotropically) weighted H\"o{}lder spaces and accompanying Liouville-type result and global Schauder-type estimates for the exterior linear problem of Ornstein-Uhlenbeck type, that are pivotal to the completion of the construction, arise from homogeneity properties of the linearized operator, which in turn lend their origin to the parabolic self-similar nature: It is the sum of homogeneous operators, with a homogeneity zero term which annihilates cones. We consider the problem of solving the equation for homogeneous functions and find good (sharp) choices for weighted H\"{o}lder norms, and then proceed for general functions with those very same spaces.

Note that the global Schauder estimates have no obvious extensions to general Laplace-type operators under the same growth rates on the coefficients, and there are counterexamples by Priola for a very similar equation (see \cite{Pr}).

It is fruitful to compare our construction with that of desingularizing, in the $H\equiv 0$ case, the intersection of a catenoid with a plane through the waist, leading to the Costa-Hoffmann-Meeks surfaces (of high genus). In that construction the plane remains flat, and one automatically gets improved power of decay of the constructed minimal surfaces back to the original plane, namely the decay rate is $1/|x|^{g+1}$ as $|x|\to\infty$. In our construction no such improvement shows up, the self-shrinkers constructed have regardless of $g$ the (likely sharp) asymptotics:
\beq\label{Asympt}
\sigma(\theta)|x|+O(|x|^{-1}),\quad |x|\to\infty.
\eeq
Another difference from the previously known constructions for minimal surfaces is that the surfaces we construct must be entropy unstable (since by \cite{CM2} the only stables ones are of the form $S^{n-k}\times\Reals^{k}, k=0,\ldots,n$), and this is another way of viewing some of the complications that arise here. However, it is from the desingularization viewpoint not presence of the instability per se that is the problem, it is the severe way in which it happens, witnessed by Equation (\ref{Asympt}): Imposing ever so much dihedral symmetry never renders it negligible.

Finally, we will mention that X.H. Nguyen via nonlinear parabolic methods has studied a related, truncated nonlinear exterior problem for the self-shrinker equation and obtained existence results (see \cite{Ng3}-\cite{Ng4}). Also, L. Wang has announced interesting existence and uniqueness result for exterior graphs with prescribed cones at infinity (see \cite{Wa}), which provides separate evidence of the dramatic change of asymptotics of the non-compact ends, i.e. that our examples are not asymptotic to planes.

After this work was completed, we learned of a preprint by X.H. Nguyen \cite{Ng5} which announces results very similar to ours.

\section{Overview of the paper}

The basic philosophy of the desingularization procedure is as follows: Consider the initial configuration of a plane intersecting a sphere through a great circle. For each $\tau$ with $\tau^{-1}=k\in\mathbb{N}$ a positive integer, define a one parameter family of surfaces $\mathcal{M}[\tau, \theta]$ that serve as approximate solutions to the self-shrinker equation. The surfaces $\mathcal{M}[\tau, \theta]$ are invariant under the action of the dihedral group with $4k$ elements, and under various normalizations converge either to the initial configuration or to Scherk's singly-periodic surface $\Sigma_0$ as the parameters $\tau$ and $\theta$ tend to zero.

On each of these surfaces, we consider graphs of small functions $u$, and produce via an incarnation of Newton's method, here Schauder's fixed point theorem, a pair $(\theta^*, u^*)$ such that the graph over $\mathcal{M}[\tau, \theta^*]$ by $u^*$ solves the self-shrinker equation exactly. Naturally, to apply the Schauder fixed point theorem, one needs to first understand the linearized equation on these surfaces, and to do this one needs to understand the linearized equation on the limits under both normalizations; that is to say on the initial configuration and on Scherk's surface. That is, we need to solve the equation $\mathcal{L} u = E$ on the initial surface $\mathcal{M}[\tau, \theta]$ with reasonable estimates, where $\mathcal{L}$ is the linearized operator for the self-shrinker equation (note that the study of this operator played an important role in \cite{CM1}-\cite{CM3}) and the function $E$ is the initial error in the self-shrinker equation on $\mathcal{M}[\tau, \theta]$.

On the pieces of the initial configuration (that is, the surfaces with boundary determined by the intersection circle), we prove that the linearized equation is always solvable with Dirichlet boundary conditions (and here we are, on the outer plane, forced to allow a dramatic change of asymptotics to include conical functions that are oscillatory in the angular variable). Near the intersection circle, the linearized equation turns out to be a perturbation of the stability operator on Scherk's surface.

The linearized equation on Scherk's surface is not solvable, with appropriate bounds on the norm of the inverse, in any bounded function space, in general, due to the persistence of a one-dimensional kernel spanned by a translational Killing field. But as long as the inhomogeneous term $E$ is ``orthogonal'' to this kernel,  we can solve the equation in a weighted H\"{o}lder space with exponential decay. The decay then allows a solution to be patched up globally to a solution on the entire initial surface. The role of the parameter $\theta$ in the surfaces $\mathcal{M}[\tau, \theta]$ is then to arrange for the initial error term $E$ to be orthogonal to 
the kernel. As $\theta$ changes, two of the pieces of $\mathcal{M}[\tau,\theta]$ move within a family of perturbed cap-shaped self-shrinkers near the round spherical caps. Note therefore that the role of the chosen $\theta^*$ in this problem is of a more technical nature (unlike for example the case of catenoidal ends for the $H\equiv0$ constructions in \cite{Ka97}, where it entails an important global change of asymptotics in itself).

The paper is structured as follows:

Section 3 sets notation and conventions for frequently used basic objects, while Section 4 discusses basic properties of the self-shrinker equation and its linearization.

In Section 5, the initial surfaces $\mathcal{M}[\tau, \theta]$ are introduced, and their basic properties -- smoothness in parameters, symmetries -- are established. 

Section 6 gives necessary estimates for the mean curvature of the desingularizing surfaces $\Sigma[\tau, \theta]$ and its variation under the $\theta$ parameter.

In Section 7, the linearized operator $\mathcal{L}$ on the curled up Scherk belt $\Sigma[\tau, \theta]$ is studied. We prove that the operator is invertible as a map between H\"{o}lder spaces with decay, modulo a one-dimensional cokernel, and we show that this cokernel can indeed be geometrically generated by varying the $\theta$ parameter.

In Section 8, we study the exterior Ornstein-Uhlenbeck problem and identify the correct weighted H\"{o}lder cone spaces which have all desired properties (such as a compact inclusion hierarchy), and in which we invert the linearized operator.

In Section 9, the patching up of solutions of the linear problem on the various pieces of the initial surfaces $\mathcal{M}[\tau, \theta]$ to a global solution is undertaken.

In Section 10, we verify the important fact that the nonlinear part of the problem closes up in the norms from Section 8, that is we prove the quadratic improvement required for Newton's method to be applicable.

Finally, in Section 11 we then complete the argument by setting up and carrying out the Schauder fixed point procedure. The Appendix at the end records various computations which were needed throughout.

\section{Notation and conventions}
Throughout $\mathbb{R}^3$ will denote Euclidean 3-space, $\vec{X}$ will denote a point in $\mathbb{R}^3$, $(x, y, z)$ the Cartesian coordinates of the point, and $\{ \vec{e}_x, \vec{e}_y, \vec{e}_z\}$ the associated standard basis, so that $\vec{X} = (x, y, z) = x \vec{e}_x + y \vec{e}_y + z \vec{e}_z$. We denote by $\mathcal{P}_{xy}$, $\mathcal{P}_{yz}$, and $\mathcal{P}_{xz}$ the $xy$-, $yz$-, and $xz$-coordinate planes respectively.

We adopt the convention in this article that for a surface $S$, all associated geometric objects and quantities will bear ``$S$'' as a subscript, with the exception of Scherk's singly-periodic surface $\Sigma_0$ and the surfaces $\Sigma[\tau, \theta]$ defined in Section \ref{initial_surfaces} . Objects associated with $\Sigma_0$ will at times simply bear the subscript ``$0$''. In most cases, the surfaces $\Sigma[\tau, \theta]$ will appear with the $\tau$ and $\theta$ arguments suppressed - so, for example, as simply $\Sigma$ - and their associated quantities will be identified without subscript. The reader should take care to distinguish subscripts from superscripts, as ``$0$'' will appear throughout the article as superscript as well.

We denote by $\vec{\nu}_S$ the Gauss map  of an oriented surface $S$. Given a function $f: S \rightarrow \mathbb{R}$ on a surface $S$, we use the shorthand $\{S : f \leq 0 \}$  to denote the set $\{p \in S : f(p) \leq 0 \} \subset S$, and likewise for ``$\geq$''. Note that under appropriate assumptions on $f$, $\{S : f \leq 0 \}$ is a smooth surface with smooth (possibly empty) boundary, and we view $\{S : f \leq 0 \}$ as inheriting all geometric quantities from $S$ -- i.e. first and second fundamental forms -- via the inclusion mapping. Also, for a function $f$, we denote by $S_f$ the normal graph of $f$ over $S$. Note that when $f$ and $S$ are class $C^{k, \alpha}$ and $f$ is sufficiently small, then $S_f$ is a $C^{k-1, \alpha}$ surface naturally parametrized by $S$.

Geometric objects defined on any  of the surfaces $\Sigma$ given in Section \ref{initial_surfaces} may be viewed as objects on $\Sigma_0$ via the map $\mathcal{Z}: \Sigma_0 \rightarrow \Sigma$.

We denote by $H^+$ the upper half plane $\{(s, z): s > 0\}$ and by $\mathcal{C}$ its quotient (a cylinder) under the action $z \mapsto z + 2 \pi$. Throughout this article, we fix a smooth, non-decreasing function $\psi_0: \mathbb{R} \rightarrow \mathbb{R}$ which vanishes on $(-\infty,1/3)$ and has $\psi_0\equiv 1$ on $(2/3,\infty)$. Also, we let $\psi[a,b]:\Reals\to[0,1]$ be
\[
\psi[a,b](s):=\psi_0\left(\fracsm{s-a}{b-a}\right),
\]
so that $\psi[a,b]$ transitions from $0$ at $a$ to $1$ at $b$.

We will for the compact pieces in our construction work in the usual weighted H\"o{}lder spaces $C^{k,\alpha}(S,g_S,f)$ on Riemannian surfaces $(S,g_S)$, defined by finiteness of the corresponding norms
\beq\label{NicosNorms}
\Big\|u: C^{k,\alpha}(S,g_S,f)\Big\|:=\sup_{x\in S}\frac{1}{f(x)}\Big\|u:C^{k,\alpha}(S\cap B(x),g_S)\Big\|,
\eeq
with weight function $f:S\to\Reals$, where $g_S$ is the metric for which the usual $C^{k,\alpha}$-norm is taken and $B(x)$ the geodesic ball of radius $1$ centered at $x$. When the metric is understood, we sometimes drop it from the notation writing $C^{k,\alpha}(S,f)=C^{k,\alpha}(S,g_S,f)$.

\section{The self-shrinker equation}
Recall that the PDE to be satisfied for a smooth oriented surface $S\subseteq\Reals^3$ to be a self-shrinker (shrinking towards the origin with singular time $T=1$) is
\beq\label{SSEq}
H_S(\vec{X})- \fracsm{1}{2} \vec{X} \cdot \vec{\nu}_S(\vec{X}) = 0,
\eeq
for each $\vec{X} \in S$, where by convention $H_S=\sum_1^n\kappa_i$ is the sum of the signed principal curvatures w.r.t. the chosen normal $\vec{\nu}_S$ (i.e. $H=2$ for the sphere with outward pointing $\vec{\nu}$). Such surfaces shrink by homothety towards the origin under flow by the (orientation-independent) mean curvature vector $\vec{H}=-H\vec{\nu}$, by the factor $\sqrt{2(1 - t)}$. In particular, we have normalized Equation (\ref{SSEq}) so that $T =  1$ is the singular time. 

The surface $\tilde{S}$ obtained by dilating  a self-shrinker $S$ about the origin by a factor of $\tau^{-1}$ satisfies the corresponding rescaled equation
\begin{equation} \label{SSEq_rescaled}
H_{\tilde{S}} (\vec{X}) - \fracsm{1}{2}\tau^2 \vec{X} \cdot \vec{\nu}_{\tilde{S}}(\vec{X}) = 0.
\end{equation} 
For a smooth  normal  variation $\vec{X}_t$ determined by a function $u$ via $X_t=X_0+tu\vec{\nu}_{\tilde{S}}$, where $\vec{X}_0$ parametrizes $\tilde{S}$, the pointwise linear change in (minus) the quantity on the left hands side in (\ref{SSEq_rescaled}) is given by the stability operator (see the Appendix, and also \cite{CM1}-\cite{CM2} for more properties of this operator)
\begin{equation} \label{SSEq_linearization}
\mathcal{L}_{\tilde{S}} u = \Delta_{\tilde{S}}  u+ |A_{\tilde{S}}|^2 u - \fracsm{1}{2} \tau^2\left(\vec{X}\cdot\nabla_{\tilde{S}} u - u  \right).
\end{equation}
Because at times we want to treat Equation (\ref{SSEq_rescaled}) as a perturbation of the mean curvature equation, we isolate the part of the linear change due to varying the mean curvature of $S$ and set
\begin{equation} \label{mc_linearization}
\mathcal{L}^0_{S} = \Delta_S + |A_S|^2.
\end{equation} 
Note that Equation (\ref{SSEq}) and its dilated version (\ref{SSEq_rescaled})  are invariant under the orthogonal group $\text{O}(3)$.

\section{The Initial Surfaces}\label{initial_surfaces}
In this section we describe in detail the construction of the initial surfaces $\mathcal{M}[\tau, \theta]$, depending on parameters $\tau$ and $\theta$ which we assume satisfy
\[
0<\tau\leq\delta_\tau,\quad|\theta|\leq\delta_\theta,
\]
throughout for appropriate constants that will later be chosen. The surfaces are approximate solutions to Equation (\ref{SSEq}), and by means of a fixed point argument we will for each small enough $\tau$ produce a function on them (for appropriately chosen $\theta$) whose graph satisfies Equation (\ref{SSEq}) exactly. The basic ingredients  are the singly periodic Scherk's singly-periodic surface $\Sigma_0$ and a family of half surfaces $\mathcal{K}[\theta]$ that are rotationally symmetric (about the $y$-axis) perturbations of the round hemisphere of radius $2$. The crucial properties of the half-surfaces $\mathcal{K}[\theta]$ are that they satisfy Equation (\ref{SSEq}) exactly, intersect the plane $\mathcal{P} = \mathcal{P}_{xz}$ at the angle $\pi/2 - \theta$ and when $\theta$ vanishes agree with the hemisphere $\mathbb{S}^2(2)\cap\{ y \geq 0 \}$.

Let $\mathcal{C}[\theta]$ denote the configuration consisting of the plane $\mathcal{P}$ together with $\mathcal{K}[\theta]$ and a copy of $\mathcal{K}[\theta]$ reflected through  $\mathcal{P}$ and let $c[\theta]$ denote their circle of intersection. For each $\tau$ with $\tau^{-1}$ an integer, the surfaces $\mathcal{M}[\tau, \theta]$ outside of a neighborhood  of $c[\theta]$ of uniformly fixed radius will agree with $\mathcal{C}[\theta]$. Inside this neighborhood they will consist, loosely speaking, of $\tau^{-1}$ fundamental domains of $\Sigma_0$, rescaled by a factor of $\tau$ that have been  ``curled'' and appropriately smoothed out to replace the singular intersection circle in the configuration. The analysis is simplified by identifying the symmetries preserved by this procedure and then imposing these from the beginning.

\begin{definition}
Let $G_\tau$ be the subgroup generated by $\omega_\tau, \xi_\tau \in \text{O}(3)$, where:
\begin{itemize}
\item[(1)] $\omega_\tau$ is the rotation about the $y$-axis by a positive angle $\pi \tau$ followed by the reflection $y \mapsto - y$.
\item[(2)] $\xi_\tau$, is the reflection through a plane $\mathcal{P}_\tau$, which is $\{z = 0\}$ rotated an angle of $(\pi/2)\tau$ around the $y$-axis.
\end{itemize}
Denote also by $\sigma_\tau=\omega_\tau^2$ the rotation about the $y$-axis by a positive angle $2 \pi \tau$.
\end{definition}

We will construct the surfaces $\mathcal{M}[\tau, \theta]$ so that they are invariant under $G_\tau$, with $\sigma_\tau$ orientation preserving and $\omega_\tau$ orientation reversing. We assume implicitly that $\tau^{-1}$ is a positive integer. These symmetries will be reflected in the analysis by working with functions on $\mathcal{M}[\tau, \theta]$ that are invariant under $\sigma_\tau$ and $\xi$ and anti-invariant under $\omega_\tau$. As the parameter $\tau\to 0$, the surfaces $\mathcal{M}[\tau, \theta]$ converge, under an appropriate renormalization,  to a surface $\Sigma [\theta]$, singly periodic  in the direction of the $z$-axis and invariant under the action of a group $G_0$, as follows:

\begin{definition}
Let $G_0$ be the group generated by the Euclidean isometries $\omega_0$ and $\xi_0$, where:
\begin{itemize}
\item[(1)] $\omega_0$ is the translation $z \mapsto z + \pi$ followed by the reflection $y \mapsto -y$.
\item[(2)] $\xi_0$ is the reflection through the plane $\{z = \pi/ 2 \}$.
\end{itemize}
Denote also by $\sigma_0=\omega_0^2$ the translation $z \mapsto z + 2 \pi$.
\end{definition}
The geometrically correct notion of symmetric functions is as in the next definition, the point being to ensure that normal graphs (using the fixed unit normal giving the orientation) over the symmetric surface inherit the symmetries.
\begin{definition}
Let $S$ be an oriented surface invariant under $G_\tau$ (resp. $G_0$). By the $G_\tau$-equivariant (resp. $G_0$-invariant) functions we will mean all $f:S\to\Reals$ such that
\[
\beta^*f=\langle\vec{\nu}_S,\beta\vec{\nu}_S\rangle f,\quad\forall\beta\in G_\tau\quad\textrm{(resp. $G_0$)}.
\]
\end{definition}

Now, recall Scherk's minimal surface $\Sigma_0$ (cf. \cite{Ka97} p. 101--106)  with angle $\fracsm{\pi}{2}$ between the asymptotic planes:
\begin{equation}\label{scherk_equation}
\Sigma_0 = \{(x,y,z)\subseteq\Reals^3| \sinh x\sinh y -\sin z= 0\}.
\end{equation}
In addition to $G_0$, the isometries of $\Sigma_0$ include reflection in the planes $\{x = y\}$ and $\{x = -y\}$. The regions $\Sigma_0 \cap \{ \pm x > 0 \}$ and $\Sigma_0 \cap \{\pm y  > 0 \}$ are graphs over  $\mathcal{P}_{xz}$ and $\mathcal{P}_{yz}$ respectively, and the symmetries of $\Sigma_0$ give that it is globally determined by the graph of a single function
\begin{equation}
f: H^+ \rightarrow \mathbb{R}.
\end{equation}
where $H^+ =  \{(s, z)| s > 0 \}$. That is, in the half space  $I = \{ (x, y, z) | x > 0 \}$) we have
\begin{equation}\notag
\Sigma_0 \cap I =  \{ (x, f(x,z), z) \}
\end{equation}
 with function $f (s,z)$ satisfying the estimate
\beq\label{f_estimate}
\|f:C^5(\{H^+: s \geq 1 \},e^{- s})\|\leq C.
\eeq
A simple rephrasing of this estimate is as follows: Let  $\Proj: \mathbb{R}^3 \rightarrow \{x\cdot y =0\}=\mathcal{P}_{xz} \cup \mathcal{P}_{yz}$ denote the nearest point projection to this closed set. Then $\Proj$  is well defined away from the planes $\{x = \pm y\}$ and  its restriction to $\Sigma_0$ satisfies the estimate $\|\Proj^{-1} - \Id\|: C^5(\{H^+ : s \geq 1\}, e^{-s}) \|\leq C$.  On $\Sigma_0$ we define the function $s$ by
\begin{equation} \label{the_function_s}
s((x, y, z)) = \max \{ |x|, |y| \}.
\end{equation}

Note that since $\Sigma_0$ is minimal, $\Sigma_0 / \langle\sigma_0\rangle$ is conformal under the Gauss map $\vec{\nu}_{\Sigma_0}$ with conformal factor $\frac{1}{2} |A_{\Sigma_0}|^2$ to the punctured sphere $\{ S^2 : x \geq 0\} \setminus \{(\pm 1, 0, 0), (0, \pm 1, 0) \}$.

Let $ \omega_0^*$ and $\xi_0^*$ denote the Euclidean isometries given by $(x, y, z) \mapsto (-x, y, -z)$ and $(x, y, z) \mapsto (x, y, - z)$, respectively. By computing the gradient of the function defining $\Sigma_0$ we obtain  the intertwining relations
 \begin{eqnarray} \label{sphere_identities}
 \vec{\nu}_{\Sigma_0} \circ \omega_0 (\vec{X}) & = & \omega_0^* \circ \vec{\nu}_{\Sigma_0}(\vec{X}), \\ \notag
 \vec{\nu}_{\Sigma_0} \circ \xi_0 (\vec{X}) & = & \xi_0 ^*\circ \vec{\nu}_{\Sigma_0}(\vec{X}).
 \end{eqnarray}
 
Thus, functions on $\Sigma_0$ that are invariant under $\xi_0$ and anti-invariant under $\omega_0$ (i.e. $G_0$-equivariant) push forward under the Gauss map to functions that are invariant under $\xi_0^*$ and anti-invariant with respect to the inversion $\omega_0^*$. Since the Gauss map  will be the fundamental tool in understanding the linear operator $\mathcal{L}_{\Sigma_0}$ on $\Sigma_0$ we record the following lemma. 

 \begin{lemma} \label{sphere_kernel}
 The kernel of the operator $\Delta_{S^2} + 2$ on the unit sphere in the space of $L^2$-functions that are invariant under $\xi_0^*$ and anti-invariant under $\omega_0^*$ is one-dimensional, spanned by the ambient coordinate function $x$.
 \end{lemma}  
\begin{proposition} \label{shrinker_caps}
 For $|\theta|\leq\delta_\theta$ with $\delta_\theta$ sufficiently small, there is a smooth one parameter family of surfaces $\mathcal{K}[\theta]$,  with the following properties:
\begin{itemize}
\item[(0)] Each $\mathcal{K}[\theta]$ satisfies Equation (\ref{SSEq}).
\item[(1)] $\mathcal{K}[0]$ is the upper hemisphere of radius $2$ and the surfaces $\mathcal{K}[\theta]$ are given as normal graphs over $\mathcal{K}[0]$.
\item[(2)] The surfaces $\mathcal{K}[\theta]$ are invariant with respect to rotations about the $y$-axis. 
\item[(3)] The boundary $\partial \mathcal{K}[\theta]$ is a circle in the plane $\mathcal{P}_{xz}$ of radius $r[\theta]$, and the inward pointing co-normal $\eta_\theta$ to $ \partial \mathcal{K}[\theta]$ at the $x$-axis satisfies 
\begin{equation} \notag
\eta_\theta \cdot \vec{e}_x = \sin (\theta).
\end{equation}
\item[(4)] There are conformal parametrizations
\begin{equation} \notag
\kappa[\theta]: \mathcal{C} \mapsto \mathcal{K}[\theta] \setminus \{ y \text{-axis} \}
\end{equation}
of the surfaces $\mathcal{K}[\theta]$, where $\mathcal{C} = H^+ / \{z \mapsto z + 2\pi \}$ is the flat cylinder of radius $1$ such that:
\begin{itemize}
\item[(i)] $ \kappa[\theta] (\{ (s, z): s = \text{const.} \}))$ is a circle with center on the $y$-axis parallel to the $xz$ plane. 
\item[(ii)] $\kappa[\theta](\{s = 0 \}) = \partial \mathcal{K}[\theta] $.
\item[(iii)]  The conformal factor is $ \varrho^2_{\kappa[\theta]} (s, z) = x^2 (s, z) + z^2(s, z)$.
\item[(iv)] There are bounds 
\begin{equation} \label{kappa_bounds}
|\nabla^{k} \kappa[\theta]|, |\nabla^{k} \dot{\kappa}[\theta]| \leq C(k)
\end{equation}
where ``$\cdot$'' denotes derivation in the $\theta$ parameter.
\end{itemize}
\end{itemize}
\end{proposition}
\begin{proof}
See Appendix.
\end{proof}

\begin{definition}
We denote by $\mathcal{K}[\tau, \theta]$ the surface $\mathcal{K}[\theta]$ dilated by the factor $\tau^{-1}$ and $\kappa[\tau, \theta]: H^+ \rightarrow \mathcal{K}[\tau, \theta]$ the map given by 
\begin{equation} \notag
\kappa[\tau, \theta] (s, z) = \tau^{-1} \kappa[\theta](\tau s, \tau z).
\end{equation}  
\end{definition}

\begin{definition} \label{bending_maps}
Let $\psi = \psi[1/2, 1]$. Then define the maps $\mathcal{B}[\tau, \theta]: \mathbb{R}^3 \rightarrow \mathbb{R}^3$ and $\mathcal{Z}[\tau, \theta]:  \mathbb{R}^3 \rightarrow \mathbb{R}^3$ by
\begin{equation} \notag
\mathcal{B}[\tau, \theta] (x, y, z) =r[\theta] \tau^{-1} e^{\tau x}( \cos \tau z,  0,  \sin \tau z) + r[\theta] y \vec{e}_y,
\end{equation}
and
\begin{equation} \notag
\mathcal{Z}[\tau, \theta] (x, y, z) =\psi(y) ( \kappa[\tau, \theta] (y, z) +r[\theta] x \vec{\nu}_{\kappa[\tau, \theta]} (y,z) ) + (1 - \psi(y))\mathcal{ B}[\tau, \theta] (x, y, z).
\end{equation}
\end{definition}

\begin{proposition} \label{bending_map_properties}
The maps $\mathcal{Z}[\tau, \theta]$ have the following properties:
\begin{itemize} 
\item[(1)] They depend smoothly on the parameters $\tau$ and $\theta$  with bounds
\begin{equation} \notag
\left| \nabla^k \mathcal{Z} [\tau, \theta] \right|, \left|\nabla^k \dot{\mathcal{Z}} [\tau, \theta] \right| \leq C\tau^{k - 1}, \quad k>1.
\end{equation}
\item[(2)]  We have that
\begin{equation} \notag
\mathcal{Z}[ \theta] : = \lim_{\tau \rightarrow 0} \mathcal{Z}[\tau, \theta]  - \tau^{-1 } r[\theta ]\vec{e}_x = r[\theta](\psi R_\theta + (1 -\psi)\Id)
\end{equation}
where $R_\theta \in \text{SO}(3)$ is the rotation determined by 
\begin{eqnarray} \notag
\vec{e}_x & \mapsto &  \cos \theta \vec{e}_x - \sin \theta \vec{e}_y \\ \notag
\vec{e}_z & \mapsto & \vec{e}_z \\ \notag
\vec{e}_y & \mapsto & \cos \theta \vec{e}_y + \sin \theta \vec{e}_x, \notag
\end{eqnarray}
 In particular $\mathcal{Z}[0]$ is globally the identity transformation. 
\end{itemize}
\end{proposition}

\begin{proof}
Claim $(1)$ follows directly from the estimates \ref{kappa_bounds} recorded in Proposition \ref{shrinker_caps}. Part (2) can be seen by applying l'H\^{o}pital's rule.
\end{proof}
We now are ready to define the ``desingularizing'' and ``initial'' surfaces, and to set notation for various distinguished subsurfaces. For technical reasons, we work with a family of cut-off Scherk surfaces that agree with the asymptotic planes $\mathcal{P}_{xz}$ and $\mathcal{P}_{yz}$ outside of a cylinder around the line $\{x=y=0\}$ and of a fixed radius proportional to $\tau^{-1}$. The reason for this is that the image of these cut-off surfaces under the maps $\tau \mathcal{Z}[\tau, \theta]$, outside of a tubular neighborhood (of fixed radius independent of $\tau$ and $\theta$) of the circle $c[\theta]$, is thus contained in the initial configuration $\mathcal{C}[\theta]$.
\begin{proposition}
We obtain ``desingularizing'' surfaces $\Sigma[\tau, \theta]$ as follows:
\begin{itemize} \notag
\item[(1)] For a constant $\delta_s>0$ to be determined later, assume $\tau\leq\delta_s$ and define the immersion $\varphi_\tau: \Sigma_0 \rightarrow \mathbb{R}^3$ by
\begin{equation} \notag
\varphi_\tau(\vec{X}) = \psi[3 \delta_s \tau^{-1}, 4 \delta_s \tau^{-1}]\vec{X}  + (1-\psi[3 \delta_s \tau^{-1}, 4 \delta_s \tau^{-1}])\Proj (\vec{X}),
\end{equation} 
where the cut-off function is evaluated at $s=s(\vec{X})$.
\item [(2)] The surface $\Sigma[\tau, \theta]$ is
\begin{equation} \notag
\Sigma[\tau, \theta] := \mathcal{Z}[\tau, \theta] \circ \varphi_\tau (\{\Sigma_0: s \leq 5 \delta_s \tau^{-1}\}),
\end{equation}
which with sufficiently small $\delta_\theta,\delta_{\tau}>0$ is well-defined, smooth and embedded for $\tau<\delta_\tau$ and $|\theta|\leq \delta_\theta$.
\end{itemize}
\end{proposition}

The set $\mathcal{T}:=\{\tau\Sigma[\tau,\theta]: \frac{4\delta_s}{\tau}\leq s\leq\frac{5\delta_s}{\tau}\}$, where the transition happens, consists of four connected components each of which is by construction a subregion of either a top/bottom spherical cap $\mathcal{K}[\theta]$ or of the plane $\mathcal{P}$. 

Considering the singular initial configuration $\mathcal{C}[\theta]$, the set $\mathcal{\mathcal{C}[\theta]\setminus\mathcal{T}}$ therefore has 5 connected components. One is the central piece containing the curve $c[\theta]$, but this singular component is now discarded and replaced by the smooth desingularizing surface $\Sigma[\tau,\theta]$ to obtain the initial surface:

\begin{definition}
The initial surface $\mathcal{M}[\tau,\theta]$ is the union of $\tau\Sigma[\tau,\theta]$ with the four components of $\mathcal{\mathcal{C}[\theta]\setminus\mathcal{T}}$ that do not contain the singular curve $c[\theta]$. 
\end{definition}

Since $\tau\Sigma[\tau,\theta]$ overlaps with $\mathcal{C}[\theta]$ in the set $\mathcal{T}$, and we have excised the set containing the singular curve $c[\theta]$, the surfaces $\mathcal{M}[\tau,\theta]$ are smooth. The constructed surfaces are orientable, but notice the topology is such that if we orient, say, the top sphere with outward pointing normal then the bottom sphere has inwards pointing normal.

\begin{proposition}
For $\delta_\theta,\delta_\tau>0$ chosen sufficiently small, the surfaces $\mathcal{M} [\tau, \theta]$ are smooth, embedded, oriented and invariant under the action of $G_\tau$.
\end{proposition}
\begin{remark}
Note that when $\tau^{-1}=k\in\mathbb{N}$, we have replaced a great circle by $2k$ Scherk handles. Hence, as computing the Euler characteristic reveals, the initial surface $\mathcal{M}[\tau,\theta]$ has topological genus $g=k-1$ and $4k$ symmetries. Thus we have:
\beq
\tau=\frac{1}{g+1}\quad\textrm{and}\quad|G_\tau|=4g+4.
\eeq
\end{remark}
% Note: Here we follow the convention $\chi(\mathcal{M})=2-2g-\#\{\textrm{boundary components}\}$.

 \begin{definition} 
We define the function $s$ on the surfaces $\Sigma[\tau, \theta]$ and $\mathcal{M}[\tau, \theta]$ as follows.
\begin{itemize}
\item[(1)] On $\Sigma[\tau, \theta]$, we take s to be the push forward by $\mathcal{Z}[\tau, \theta] \cdot \phi_\tau$ of the function $s$ defined on $\Sigma_0$. 
\item[(2)] s is then extended  continuously to a constant on the remainder of $\mathcal{M}[\tau, \theta] \supset \Sigma[\tau, \theta]$.
\end{itemize}

\end{definition}
\begin{remark} \label{object_ident}
The reader will note that the surfaces $\Sigma[\tau, \theta]$ are by construction diffeomorphic to $\{ \Sigma_0: s  \leq 5 \delta_s \tau^{-1}\}$ under the map $\mathcal{Z}[\tau, \theta] \circ \varphi_\tau$. We will, throughout this article, identify functions, tensors, and operators on $\Sigma[\tau, \theta]$ with their pull-backs by  $\mathcal{Z}[\tau, \theta] \circ \varphi_\tau$, and vice versa. 
\end{remark}

\section{Geometric quantities on the initial surfaces}

\begin{proposition} \label{H_est}
Let $0<\gamma<1$. Then on $\{\Sigma[\tau,\theta]: s \geq 1\}$ we have:
\begin{equation} \notag
\left\| H_\Sigma - \fracsm{1}{2} \tau^2 \vec{X} \cdot \vec{\nu}_\Sigma: C^2(\{\Sigma[\tau,\theta]: s \geq 1\}, e^{-\gamma s})\right\| \leq C \tau,
\end{equation}
and
\begin{equation} \notag
\left\|\frac{\partial}{\partial \theta} \Big\{ H_{\Sigma}  - \fracsm{1}{2} \tau^2 \vec{X} \cdot \vec{\nu}_{\Sigma}\Big\}: C^1(\{\Sigma[\tau,\theta]: s \geq 1\}, e^{- \gamma s})\right\| \leq C \tau.
\end{equation}
for $\tau>0$ sufficiently small.
\end{proposition}

\begin{proof}
In the following we let $\delta=\delta_{ij}$ denote the flat standard metric on the upper half plane $H^+$. Note that the surface $\Sigma[\tau, \theta]\cap\{y \geq 1\}$ has the parametrization  $\varphi: \{H^+ : y \geq 1 \} \rightarrow \mathbb{R}^3$ given by
 \begin{equation} \notag
 \varphi (s, z)   = \kappa[\tau,\theta] (s, z)+  \psi[4 \delta_s \tau^{-1}, 3 \delta_s \tau^{-1} ](s) f(s, z) r[\theta]\vec{\nu}_{\kappa}(s,  z),
 \end{equation} 
We will in the rest of this proof denote $\kappa = \kappa[\tau, \theta]$. When $s \geq 4 \delta_s / \tau$ the estimates are trivially satisfied since $\varphi \equiv \kappa[\tau, \theta] $ in this region. For $s$ belonging to the interval $[3 \delta_s/ \tau, 4 \delta_s / \tau ]$, we have that
\begin{equation} \notag
 H_{\Sigma} - \fracsm{1}{2} \tau^2 \vec{X} \cdot \vec{\nu}_{\Sigma} = -\Delta_\mathcal{K} \hat{f}-|A_\mathcal{K}|^2 \hat{f} + \fracsm{1}{2} \tau^2\left( \nabla_\mathcal{K} \hat{f} \cdot \vec{X} - \hat{f} \right) + Q_{\hat{f}}+\fracsm{1}{2}\tau^2\vec{X}\cdot\vec{Q}_{\hat{f}},
 \end{equation}
 where $\hat{f} =  \psi[4 \delta_s \tau^{-1}, 3 \delta_s \tau^{-1} ](s) f(s,z)$ and $Q_{\hat{f}}$ and $\vec{Q}_u$ denote terms that are at least quadratic in $\hat{f}$ and its derivatives. In this region, we may (since $\gamma<1$) arrange that $e^{-s} < \tau e^{- \gamma s}$ by taking $\tau$ sufficiently small in terms of $\gamma$. The estimate then follows by observing that $|\nabla_\delta^{k} \hat{f}| \leq C e^{-s}$, $k = 0, 1, 2$,  that $\varrho_{\kappa}^{-2} \Delta_\delta = \Delta_\mathcal{K}$, and that both $\varrho_{\kappa}$ and $|A_\mathcal{K}|^2 $ are uniformly bounded in this region.
 
We now treat the case $s \leq 3 \delta_s / \tau$ as follows. Since $\{\Sigma_0 : y \geq 1\}$ is a graph over $H^+$ which is itself minimal, and since dilations preserve minimality, we have from the variation formula (\ref{MinimalOnMinimal}) in the Appendix the relation 
\begin{equation} \label{MinimalOnMinimal2nd}
\Delta_\delta f = r_\theta Q _{f}. 
\end{equation}
We then estimate the error term on $\Sigma=\{ \Sigma[\tau,\theta]: s \leq 3\delta_s / \tau \}$, using that it is a  graph over $\mathcal{K} = \mathcal{K}[\tau, \theta]$, as follows:
\begin{align} \notag
&H_\Sigma - \fracsm{1}{2} \tau^2 \vec{X} \cdot \vec{\nu}_{\Sigma}  =  -r_\theta\mathcal{L}_{\mathcal{K}} f + Q_{{r_\theta} f}+\fracsm{1}{2}\tau^2\vec{X}\cdot\vec{Q}_{r_\theta f} \\ \notag
& \quad\quad=    -r_\theta\Delta_{\mathcal{K}} f -  |A_{\mathcal{K}}|^2 {r_\theta} f + \fracsm{1}{2} \tau^2 r_\theta\left(\vec{X}\cdot\nabla_{\mathcal{K}}f-  f \right) + Q_{{r_\theta} f} + \fracsm{1}{2}\tau^2\vec{X}\cdot\vec{Q}_{r_\theta f}\\ \notag
& \quad\quad=    -|A_{\mathcal{K}}|^2 {r_\theta} f +  \fracsm{1}{2} \tau^2 r_\theta \left( \vec{X}\cdot\nabla_{\mathcal{K}}  f - f \right) + r_\theta^2(Q_f - \varrho_{\kappa}^{-2}Q_{f})+\fracsm{1}{2}\tau^2\vec{X}\cdot\vec{Q}_{r_\theta f},
\end{align}
where in the last equality we have used (\ref{MinimalOnMinimal2nd}).

Note that as a consequence of the estimates for $\mathcal{Z}[\tau, \theta]$ recorded in (\ref{bending_map_properties}) the terms $|A_{\mathcal{K}}|^2 {r_\theta} f$ and $\fracsm{1}{2} \tau^2 r_\theta ( \vec{X}\cdot\nabla_{\mathcal{K}}  f - f )$ appearing above and their variations by $\theta$  satisfy the desired estimates, so it remains to estimate the terms $ \mathcal{R} := r_\theta^2(Q_f - \varrho_{\kappa}^{-2}Q_{f})$.  At  $\tau=0$ one has that $\mathcal{R}\equiv 0$, and since one may verify that the map $(\tau, \theta) \mapsto \mathcal{R}(\cdot)$ is $C^1$ in the parameters $\tau\geq 0$ and $\theta$ as a map into $C^2(H^+ ,\delta,e^{- \gamma s})$, we get the claimed estimates by one-sided Taylor expansion. 
 \end{proof}

\section{The linearized equation away from the end}

\begin{definition} \label{the_function_w}
Set $\Sigma[\theta] = \mathcal{Z}[\theta] (\Sigma_0)$ and let the function $\dot{H}_{\Sigma[\theta]}:\Sigma[\theta]\to\Reals$ denote the variation under $\theta$ of the mean curvature $H_{\Sigma[\theta]}$ of $\Sigma[\theta]$, that is for all $x_0\in\Sigma_0$
\[
\dot{H}_{\Sigma[\theta]}\circ \mathcal{Z}[\theta](x_0):=\frac{\partial}{\partial\theta}\Big[H_{\Sigma[\theta]}\circ \mathcal{Z}[\theta](x_0)\Big]
\]
Then the function $w: \Sigma[\tau,\theta] \rightarrow \mathbb{R}$  is given by
\begin{equation} \notag
w[\tau, \theta] =  \dot{H}_{\Sigma[\theta]} \circ \mathcal{Z}[\theta] \circ \mathcal{(Z\circ \varphi_\tau})^{-1}[\tau, \theta].
\end{equation}
where we are viewing $\mathcal{Z}[\tau, \theta] \circ \varphi_\tau$ as a diffeomorphism of $\{\Sigma_0: s \leq 5 \delta_s/ \tau \}$ onto $\Sigma[\tau, \theta]$. 
\end{definition}

\begin{lemma} \label{w_properties}
The function $w$ has the following properties:
\begin{itemize}
\item [(1)] $w$ is supported on $\{\Sigma[\tau,\theta]: s \leq 1\} $.
\item [(2)] The estimate
 \begin{equation}  \notag
\left\| \frac{\partial}{\partial \theta} \Big\{H_\Sigma - \fracsm{1}{2} \tau^2 (\vec{X} \cdot \vec{\nu}_\Sigma) \Big\}- w: C^{1}(\Sigma, g, e^{- \gamma s}) \right\| \leq C\tau,
\end{equation}
holds for all sufficiently small $\theta$ and $\tau$.
\item[(3)] When $\tau=0$ and $ \theta = 0$ it holds
\begin{equation} \notag
\int_{\Sigma_0/ \langle \sigma_0\rangle} w_0 (\vec{e}_x \cdot \vec{\nu}) d\mu_{\Sigma_0} = 8 \pi.
\end{equation}
where $w_0 \equiv w[0, 0]$.
\end{itemize}
\end{lemma}
\begin{proof}
(1) and (2) follow directly from Definition \ref{the_function_w}  and Proposition \ref{H_est}.  To see (3), set $S_c = \{\Sigma_0: s \leq c \} / \langle\sigma_0\rangle$. We then we have  
\[
\int_{S_c } w (\vec{e}_x \cdot \vec{\nu}_{\Sigma_0}) = 
\int_{S_c}  u \mathcal{L}_{\Sigma_0} (\vec{e}_x \cdot \vec{\nu}_{\Sigma_0})  + \int_{\partial S_c} \Big[(\vec{e}_x \cdot \vec{\nu}) (\nabla u \cdot \vec{\eta}) - u \vec{\eta}\cdot\nabla (\vec{e}_x \cdot \vec{\nu}) \Big]
\]
where $\nabla = \nabla_{\Sigma_0}$, $\vec{\eta}$ is the co-normal at the boundary of $S_c$, and $u = \frac{\partial}{\partial \theta} \mathcal{Z}[\theta]\big|_{\theta = 0} \cdot \vec{\nu}_{\Sigma_0}$, so that $\mathcal{L}_{\Sigma_0} u = w_0$. The claim then follows by taking $c$ to $\infty$ and noting that $| \nabla( \vec{e}_x \cdot \vec{\nu}_{\Sigma_0})(s, z)| \leq C e^{-s} $ , $|\vec{\eta} - \vec{e}_y| \leq C e^{-c}$, and $|(\nabla u(s, z) - 2\vec{e}_y)| \leq C e^{-s}$.
\end{proof}

By Proposition \ref{H_est}, the quantity $ E = H_\Sigma - \fracsm{1}{2} \tau^2 ( \vec{X} \cdot \vec{\nu}_\Sigma )$ and its variations under $\theta$ lie in the weighted H\"o{}lder spaces $C^{0, \alpha} (\Sigma, g, e^{- \gamma s})$. The symmetries of $\Sigma$ give that $E$ is  $G_\tau$-equivariant, and that its pull-back to $\Sigma_0$ by $\mathcal{Z}$ is $G_0$-equivariant.  For the remainder of this article, all functions defined on $\Sigma[\tau, \theta]$ are assumed to be invariant under the symmetry group $G_\tau$. 

\begin{proposition} \label{invertibility}
Given any $E \in \mathcal{C}^{0, \alpha} (\Sigma,  g, e^{-\gamma s}) $, there is a constant $b = b_E$ and a function $v = v_E$ such that
\begin{eqnarray} \notag
\mathcal{L}_\Sigma v &  = & E - b w, \\ \notag
v  & = & 0, \quad\text{ on }\quad \partial \Sigma,
\end{eqnarray}
\begin{equation} \notag
|b|,\:\|v: C^{2, \alpha}(\Sigma, g, e^{-\gamma s})\| \leq C \|E: C^{0, \alpha}\left(\Sigma , g,  e^{-\gamma s}\right)\|.
\end{equation}
Moreover, the pair $(v_E, b_E)$ depends continuously on the parameters $\tau$ and $\theta$ (see Remark (\ref{object_ident}))
\end{proposition}

We prove first Proposition \ref{invertibility} in the limiting case $\tau = 0$, $\theta = 0$, and handle the general case as a perturbation. 

\begin{proposition} \label{unperturbed_invertibility}
Given any $E \in \mathcal{C}^{0, \alpha} (\Sigma_0, g_0, e^{-\gamma s}) $, there is a constant $b = b_E$ and a function $v = v_E$ such that
\begin{equation} \label{scherk_eq}
\mathcal{L}_{\Sigma_0} v = E - b w_0,
\end{equation}
and such that 
\begin{equation} \notag
|b|,\: \|v: \mathcal{C}^{2, \alpha}(\Sigma_0,g_0,  e^{-\gamma s})\| \leq C \|E: \mathcal{C}^{0, \alpha}(\Sigma_0, g_0,  e^{-\gamma s})\|
\end{equation}
\end{proposition}
\begin{proof}
Let $E \in C^{0, \alpha} (\Sigma_0, e^{- \gamma s})$ be a given $G_0$-equivariant function, and assume for the moment that $E$ is supported on $\{\Sigma_0 : s \leq a \}$ where $a>1$ is a large constant.  Recall that the Gauss map 
\begin{equation} \notag
\nu_{\Sigma_0}: \Sigma_0 \rightarrow \mathbb{S}^2 
\end{equation}
is a conformal covering which descends to a diffeomorphism from $\Sigma_0 / \sigma_0$ onto the punctured sphere $\mathbb{S}^2 \setminus \{(\pm 1, 0, 0), (0, \pm 1, 0) \}$, with the four punctures corresponding to the four asymptotic ends of $\Sigma_0$.  The function $\bar{E} = \vec{\nu}_{\Sigma_0 *} (E / |A_{\Sigma_0}|^2) \in L^2 (\mathbb{S}^2, g_{\mathbb{S}^2})$ is then well-defined and satisfies 
\begin{equation}  \notag
\| \bar{E} \|_{L^2 (\mathbb{S}^2)} \leq C \| E \|_{C^{0, \alpha} (\Sigma_0, e^{- \gamma s})}
\end{equation}
where the constant $C$ depends on $a$. It is easily verified that since $E$ is $G_0$-equivariant, the function $\bar{E}$ satisfies the identities (\ref{sphere_identities}), which then give that $\bar{E}$ is $L^2$ orthogonal to the functions $y$ and $z$ on $\mathbb{S}^2$ from Lemma \ref{sphere_kernel}. Now, (3) in Lemma \ref{w_properties} gives that 
\begin{equation} \notag
\int_{\mathbb{S}^2} \bar{w} x = 8 \pi,
\end{equation}
where $\bar{w} = \vec{\nu}_{\Sigma_0 *} (w_0 / |A_{\Sigma}|^2)$. Thus, we may find a constant $b$ such that $\bar{E} - b \bar{w} $ is $L^2$-orthogonal to $x$. We then get a function $v: \mathbb{S}^2 \rightarrow \mathbb{R}$ satisfying 
\begin{equation} \notag
(\Delta_{\mathbb{S}^2} + 2 ) v = \bar{E} - \bar{w}
\end{equation}
and the identities (\ref{sphere_identities}),  from which we conclude that $v(1, 0, 0) = - v (-1, 0, 0)$, while $v (0, \pm 1, 0 ) = 0$.  Define then the $G_0$-equivariant function $u: \Sigma_0 \rightarrow \mathbb{R}$ by
\begin{equation} \notag
u = \vec{\nu}_{\Sigma_0}^* (v - v(1, 0, 0) x).
\end{equation}
We then get immediately that $u$  satisfies 
\begin{equation} \notag
\mathcal{L}_{\Sigma_0} u = E - b w_0.
\end{equation}
That $u$ has the appropriate decay, i.e. lies in the space $C^{2, \alpha} (\Sigma_0, g_0, e^{- \gamma s})$, follows by observing that the operator $\mathcal{L}_{\Sigma_0}$ is asymptotically a perturbation of the Laplace operator on the flat cylinder $\mathcal{C}$, for which the decay estimates hold. To conclude the proof, note that we may reduce to the case that $E$ is supported in $\{\Sigma_0 : s \leq a \}$ as follows:  Recall  that each component of $\{\Sigma_0 : s \geq a\}$ is given by the graph of a small function $f: H^+ \rightarrow \mathbb{R}$ with $f$ satisfying (\ref{f_estimate}). For $a$ sufficiently large, the operator $\mathcal{L}_{\Sigma_0}$ on $\{\Sigma_0 : s \geq a - 1\}$  is then a perturbation of the Laplace operator $\Delta_{H^+}$ on the flat half cylinder $H^+$. Proposition \ref{laplacian_on_cylinders} then gives a function $u'$ on  $\{ \Sigma_0 : s \geq a \}$ satisfying 
\begin{eqnarray} \notag
\mathcal{L}_{\Sigma_0} u' & = & E, \\ \notag
u' & = & c,\quad\text{ on }\quad \partial \{\Sigma_0 : s \geq a - 1\}.
\end{eqnarray}
for a constant $c$ with $|c| \leq C \| E\|$. Define the smooth cutoff function $\psi = \psi[a - 1, a]$. We then get that $\psi u'$ is defined on all of $\Sigma_0$ and satisfies
\begin{equation} \notag
\mathcal{L}_{\Sigma_0} (\psi u' ) = \psi{E} + \mathcal{E}
\end{equation} 
where $\mathcal{E}$ is an error term introduced by smoothing out $u'$ on the boundary of $\{ \Sigma_0: s \geq a -1\}$. The function $F = (1- \psi)E  - \mathcal{E}$ is then supported on $\{ \Sigma_0: s \leq a\}$. This concludes the proof.
\end{proof}
\begin{remark}
The reader will note that the highly symmetric nature of our construction, in contrast with the general situation and in particular the construction in   \cite{Ka97}, allow us to obtain a solution with the appropriate decay with a single parameter.

In particular, the function $v$ satisfying the equivalent problem on the sphere has opposite values at $(\pm 1, 0, 0)$, which allows simultaneous cancellation of both values by a single multiple of the kernel element $x$. 
\end{remark}
\begin{corollary}
Given 
\begin{equation} \notag
E \in C^{0, \alpha}(\{\Sigma_0: s  \leq 5 \delta_s \tau^{-1} \},  g_0, e^{- \gamma s}),
\end{equation}
there is a  constant $b \in \mathbb{R}$ and a function
\begin{equation} \notag
v \in C^{0, \alpha}(\{\Sigma_0: s  \leq 5 \delta_s \tau^{-1} \}, g_0, e^{- \gamma s})
\end{equation}
such that 
\begin{eqnarray} \notag
\mathcal{L}_{\Sigma_0} v & = & E - b w_0 \text{ on } \{\Sigma_0: s \leq 5 \delta_s \tau^{-1} \} \\ \notag
v & = & 0 \text{ on } \partial  \{\Sigma_0: s \leq 5 \delta_s \tau^{-1} \}
\end{eqnarray}
with the bounds 
\begin{equation} \notag
\|v : C^{2, \alpha}(\{\Sigma_0: s  \leq 5 \delta_s \tau^{-1} \} , g_0, e^{- \gamma s}) \| \leq C \|E : C^{0, \alpha}(\{\Sigma_0: s  \leq 5 \delta_s \tau^{-1} \} , g_0, e^{- \gamma s} )\|
\end{equation}
and
\begin{equation} \notag
|b| \leq C \|E : C^{0, \alpha}(\{\Sigma_0: s  \leq 5 \delta_s \tau^{-1} \} , g_0, e^{- \gamma s} )\|.
\end{equation}
\end{corollary}
\begin{proof}
First, we apply Proposition \ref{unperturbed_invertibility} to obtain a function $v_1$ satisfying (\ref{scherk_eq}). Now, note that for a large constant $a > 0$, the operator $\mathcal{L}_{\Sigma_0} = \Delta_{\Sigma_0} + |A_{\Sigma_0}|^2$ on $\{\Sigma_0: a \leq s \leq 5 \delta_s \tau^{-1} \}$ is a perturbation of the Laplacian on a long cylinder. This allows us (see Proposition \ref{laplacian_on_cylinders}) to solve the following  Dirichlet problem, with $ \partial_a$ and $\partial_\tau$ denoting the boundary components of $\{\Sigma_0: a \leq s \leq 5 \delta_s \tau^{-1} \}$ in the obvious way: 
\begin{eqnarray}
\mathcal{L}_{\Sigma_0} v_2 & = & 0 \\ \notag
v_2  & = & v_1 + c_1 \text{ on } \partial_\tau \\ \notag
v_2 & = & 0 \text{ on } \partial _a
\end{eqnarray}
 with the bounds
\begin{eqnarray} \notag
|c_1| , \|v_2: C^{2, \alpha}(\{\Sigma_0 : a \leq s \leq 5 \delta_s \tau^{- 1} \}, g_0)\| & \leq & C \|v_1: C^{2, \alpha} (\partial_\tau, g_0) \| \\ \notag
& \leq & C e^{- 5 \gamma \delta_s \tau^{-1}} \|E : C^{0, \alpha}(\{\Sigma_0: s  \leq 5 \delta_s \tau^{-1} \} , g_0, e^{- \gamma s} )\|.
\end{eqnarray}
The function $ v = v_1 -  \psi[a, a + 1] (v_2 - c_1)$ then solves
\begin{eqnarray} \notag
\mathcal{L}_{\Sigma_0} v  & = & E - bw_0 + \mathcal{E} \text { on } \{\Sigma_0: s \leq 5 \delta_s \tau^{-1}\} \\ \notag
v  & = & 0 \text{ on } \partial \{\Sigma_0: s \leq 5 \delta_s \tau^{-1} \}.
\end{eqnarray}
for an error term $\mathcal{E}$ and has the required bounds on the norm,  and by taking $\tau$ sufficiently small and using that $|A_{\Sigma_0} |^2 < C e^{-s}$ (a consequence of (\ref{f_estimate})) we get that
\begin{equation} \notag
\|\mathcal{E} : C^{0, \alpha}(\{\Sigma_0: s  \leq 5 \delta_s \tau^{-1} \} , g_0, e^{- \gamma s} )\| \leq 1/2 \|E : C^{0, \alpha}(\{\Sigma_0: s  \leq 5 \delta_s \tau^{-1} \} , g_0, e^{- \gamma s} )\|.
\end{equation}
We then iterate this process to obtain an exact solution.
\end{proof}

We now prove Proposition \ref{invertibility} in full generality.  
\begin{proof}
Recall that $\mathcal{Z} \circ \varphi_\tau: \{\Sigma_0 : s \leq 5 \delta_s / \tau \} \rightarrow \Sigma$ is a diffeomorphism. By referring to the derivative bounds on the maps $\mathcal{Z}[\tau, \theta]$ recorded in Proposition \ref{bending_map_properties} it is clear that we can arrange so that  
\begin{equation} \notag
\| g_{\Sigma_0}- \mathcal{(Z \circ \varphi_\tau)}^* g_\Sigma: C^{2, \alpha}(\Sigma_0, g_0) \| < \epsilon \end{equation}
by choosing the constant $\delta_s$ sufficiently small for arbitrary positive $\epsilon$.
Now, by choosing $a$ sufficiently large and $\tau$ sufficiently small, we can arrange that 
\begin{equation}
\mathcal{(Z\circ \varphi_\tau)}^* |A_\Sigma|^2, |A_{\Sigma_0}|^2 < \epsilon
\end{equation}
on $\{\Sigma_0 : a \leq s \leq 5 \delta_s / \tau \}$. It follows that the operator norm of $\mathcal{(Z \circ \varphi_\tau)}^*\mathcal{L}_\Sigma -  \mathcal{L}_{\Sigma_0}:C^{2, \alpha} (\Sigma_0) \rightarrow C^{0, \alpha} (\Sigma_0) $ can be made arbitrarily small. The proposition then follows by formally treating $\mathcal{(Z\circ \varphi_\tau)}^*\mathcal{L}_\Sigma$ as a perturbation of $\mathcal{L}_{\Sigma_0}$. 

 \end{proof}

\begin{lemma}\label{H_decomp}
For any $\gamma\in(0,1)$ there exists $C=C(\gamma)$ such that
\begin{align*}
&\left\|H_\Sigma-\fracsm{1}{2}\tau^2\vec{X}\cdot\vec{v}_\Sigma-\theta w :C^{0,\alpha}(\Sigma, g_{\Sigma},e^{-\gamma s})\right\|\\
&\quad\leq C(\tau+|\theta|^2),
\end{align*}
where $H_\Sigma$ is the mean curvature of $\Sigma$.
\end{lemma}
\begin{proof}
This is a consequence of the smooth dependence of the surfaces $\Sigma$ on the parameters $\theta, \tau$ and the definition of $w$.
\end{proof}

\section{An exterior linear problem of Ornstein-Uhlenbeck type}

On a flat plane $\mathcal{P}$ through the origin, with the induced standard Euclidean metric, the Dirichlet problem for the linearized operator $\mathcal{L}_\mathcal{P}$ in (\ref{SSEq_linearization}) at unit scale becomes:
\beq\label{EqPlane}
\begin{cases}
&\mathcal{L}_\mathcal{P}u=\Delta u-\fracsm{1}{2}\big(\vec{X}\cdot\nabla u-u\big)=E,\\
&u_{\mid\partial\Omega}=0.
\end{cases}
\eeq
for $u:\Omega\to\Reals$, where the domain $\Omega = \Reals^2\setminus B_R(0)$ is the exterior of a disk with radius $R\simeq 2$. The Laplacian and gradient are taken with respect to the standard Euclidean metric on the plane. The function $E$ is implicitly assumed to be $G_\tau$-equivariant.

The operator $\mathcal{L}_\mathcal{P}$ is of Ornstein-Uhlenbeck type (such operators are related to Brownian motion and number operators in quantum mechanics). It is of course clear that the local theory for this equation is classic, using for example standard Schauder estimates. On the non-compact exterior domain however, with such fast growth on the gradient term, there is generally no reasonable global elliptic theory available (see for example the counterexamples \cite{Pr}) and it is not a priori clear even what spaces to study the problem in. There exists in fact a vast literature on Ornstein-Uhlenbeck operators for various restrictive assumptions on the coefficients and corresponding choices of function spaces (see for example \cite{CV} and \cite{DL}), but since remarkably there is nothing in the literature that is adequate for our construction, we must develop our theory from scratch.

Firstly, note that the connection with the stability operator as a minimal surface in the Gaussian metric (see (\ref{StabPlane}) in the Appendix), is via the following conjugation identity,
\beq\label{Harmonic}
\mathcal{L}_\mathcal{P}u=\Delta u-\fracsm{1}{2}(\vec{X}\cdot\nabla-1)u=e^{|x|^2/8}\Big(\Delta -\frac{|x|^2}{16}+1\Big)e^{-|x|^2/8}u,
\eeq
where the exponential functions act by multiplication.

The operator in the parentheses in (\ref{Harmonic}) is of course nothing but the Hamilton operator for the two-dimensional quantum harmonic oscillator, plus a constant. Rescaling coordinates, it has the expression
\beq
\hat{H}=\fracsm{1}{2}\Delta-\fracsm{1}{2}|x|^2+2.
\eeq

This connection to the harmonic oscillator turns out to be about as misleading as it is helpful, for as we will see below, it is certainly not a natural point of departure for our applications, because of the involved conjugation with the Gaussian densities.

We get however from (\ref{Harmonic}) the following elementary lemma. The notation $H^s(\Reals^2)$ refers to the Sobolev space of functions with $s$ derivatives in $L^2(\Reals^2)$.

\begin{lemma}
Given $G_\tau$-equivariant $E\in e^{|x|^2/8}L^2(\Reals^2)$ and assuming $\tau\leq\fracsm{1}{3}$, there is a unique $G_\tau$-equivariant $u\in e^{|x|^2/8} H^2(\Reals^2)$ such that $\mathcal{L}_\mathcal{P}u=E$.

Furthermore, there is a uniform constant $C>0$ such that
\beq\label{LinearGrowth}
|u(x)|\leq C\Big(\sup_{x\in\Reals^2}(1+|x|)|E(x)|\Big)(1+|x|),\quad x\in\Reals^2,
\eeq
for all $E\in C^0(\Reals^2)$ s.t. $\sup_{x\in\Reals^2}(1+|x|)|E(x)|<\infty$ (hence $E\in e^{|x|^2/8}L^2(\Reals^2)$).

The same statements hold if we replace $\Reals^2$ by $\Omega = \Reals^2\setminus B_R(0)$ and add the condition $u_{\mid \partial\Omega}=0$.
\end{lemma}

\begin{proof}
Since the $L^2$-eigenvalues of $\hat{H}$ are $\lambda_{(n_1,n_2)}(\hat{H})=n_1+n_2+1$, for $n_i\geq 0$, and the well-known $L^2$-basis for $\hat{H}$ consists of Hermite functions, the $e^{|x|^2/8}L^2(\Reals^2)$-kernel of $\mathcal{L}_\mathcal{P}$ thus corresponds to the first excited eigenmodes,
\[
\ker\mathcal{L}_\mathcal{P}=\textrm{span}\{x_1,x_2\},
\]
which thus disappears under the assumption of $G_\tau$-equivariance (given we insert at least $2\tau^{-1}=2k\geq2$ handles). Hence there is a well-defined inverse map $\mathcal{L}_\mathcal{P}^{-1}:L^2(\Reals^2)\to H^2(\Reals^2)$, which by isometry invariance of $\mathcal{L}_\mathcal{P}$ preserves the imposed symmetries.

If we consider the disk $B_{\sqrt{17}}(0)=\{|x|^2\leq17\}$, then if $v\in H^2(\Reals^2)$ satisfies $\hat{H}v\geq 0$ and $v\leq 0$ on $\partial B_{\sqrt{17}}(0)$, we conclude the simple maximum principle result that $v\leq 0$ on $\Omega=\Reals^2\setminus B_{\sqrt{17}}(0)$. This is standard, but we briefly sketch the proof. Namely, let $w:=\max(0,v)$ so that
\beq
w\Delta v\geq (\fracsm{1}{2}|x|^2-2)w^2\geq 0
\eeq
Then by Green's first identity, which is justified since $v\in H^2(\Reals^2)$ and $w\in H^1(\Reals^2)$,
\beq
-\int_{\Reals^2\setminus B_{\sqrt{17}}(0)}|\nabla w|^2\geq 0,
\eeq
where we used $w_{\mid\partial B_{\sqrt{17}}(0)}=0$. Thus $w=0$ which proves the claim.

We now take, for numbers $A, B>0$ to be determined below, the test functions 
\[
v(x_1,x_2):=e^{-|x|^2/8}\big(u-Ax_1+\frac{B}{2}\frac{x_1}{|x|^2}\big).
\]
We consider a fundamental domain $\theta\in[-\pi/k,\pi/k]$ positioned inside the support of the test functions and compute:
\begin{align*}
\hat{H}v&=e^{-|x|^2/8}\Big(E+B\frac{x_1}{|x|^2}\Big)\geq e^{-|x|^2/8}\left(B-\left|\frac{|x|E}{\cos(\frac{\pi}{k})}\right|\right)\frac{x_1}{|x|^2}\\
&\geq e^{-|x|^2/8}\left(B-2\sup_{x\in\Reals^2}(1+|x|)|E(x)|\right)\frac{x_1}{|x|^2},
\end{align*}
where we have used that $k\geq 3$, so that $\cos(\pi/k)\geq \frac{1}{2}$. Thus we get that picking $B:=2\sup_{x\in\Reals^2}(1+|x|)|E(x)|$ ensures $\hat{H}v\geq0$ (on the fundamental domain). By picking $A$ large depending linearly on $B$ and on $\|u_{\mid \partial B_{\sqrt{17}}}\|_\infty$, we arrange $v\leq0$ on (a fundamental domain of) $\partial B_{\sqrt{17}}$, and hence the result follows by the above maximum principle combined with the estimate
\begin{align*}
\|u_{\mid \partial B_{\sqrt{17}}}\|_\infty&\leq C\|e^{-|x|^2/8}u\|_{H^2(B_5(0))}\leq C \|e^{-|x|^2/8}E(x) \|_{L^2(\Reals^2)}\\
&\leq C \left(\int_{\Reals^2}\frac{e^{-|x|^2/4}}{(1+|x|)^{2}}dx\right)^{1/2}\sup_{x\in\Reals^2}(1+|x|)|E(x)|,
\end{align*}
using the Sobolev inequality on the larger disk $B_5(0)$. Hence the estimate (\ref{LinearGrowth}) follows. The argument in the case with $\Omega$ instead of $\Reals^2$ is similar.
\end{proof}

However nice such simple lemmas may appear, the truth is that the spaces $e^{|x|^2/8}L^2(\Omega)$ are not well-suited for our geometric analysis purposes, in particular they do not have any good compact embedding properties, because of the Gaussian (and linear) growth involved. What we would like is to separate out the conical asymptotics and obtain sharp, uniform control in adequate weighted spaces, with second order derivative bounds, in such a way that we can proceed with our geometric construction. To accomplish this, we first introduce in the next section the appropriate new cone spaces.

\subsection{H\"o{}lder cone spaces for the exterior problem}
In this section we define the weighted H\"{o}lder spaces suitable for working with homogeneous functions. Note that these are different from the standard spaces considered in Equation (\ref{NicosNorms}), although they could be naturally rephrased as such with a different metric (in fact the pull-back metric under the projection from any fixed symmetric cone) on the plane.
\begin{definition}[Homogeneously weighted H\"older spaces]
We define the appropriate weighted spaces of H\"older functions for decay rate $k\in\mathbb{N}$,
\[
C_\mathrm{hom}^{0,\alpha}(\Omega, |x|^{-k})=\{f\in C^{0,\alpha}_\mathrm{loc}(\Omega):\|f:C_\mathrm{hom}^{0,\alpha}(\Omega, |x|^{-k})\|<\infty\},
\]
with norms
\[
\|f:C_\mathrm{hom}^{0,\alpha}(\Omega, |x|^{-k})\|:=[f]_{\Omega,-k-\alpha}+\sup_{x\in\Omega} |x|^k|f(x)|,
\]
where the weighted H\"older coefficients of decay rate $-k-\alpha$ are defined as:
\[
[f]_{\Omega,\alpha,-k-\alpha}:=\sup_{x,y\in\Omega}\frac{1}{|x|^{-k-\alpha}+|y|^{-k-\alpha}}\frac{|f(x)-f(y)|}{|x-y|^\alpha}.
\]

We then let:
\begin{align*}
&C_\mathrm{hom}^{2,\alpha}(\Omega, |x|^{-1}):=\{f\in C_\mathrm{loc}^{0,\alpha}(\Omega):D_{\beta}f\in C_\mathrm{hom}^{0,\alpha}(\Omega,|x|^{-1}),\: |\beta|\leq 2\},
\end{align*}
where $\beta$ ranges over all multiindices, with norm given by
\beq
\|f:C_\mathrm{hom}^{2,\alpha}(\Omega, |x|^{-1})\|^2:=\sum_{|\beta|\leq 2}\|D_\beta f:C_\mathrm{hom}^{0,\alpha}(\Omega, |x|^{-1})\|^2.
\eeq
\end{definition}

\begin{definition} \label{aniso_spaces}
The anisotropically homogeneously weighted H\"o{}lder spaces are the following:
\[
C_\mathrm{an}^{2,\alpha}(\Omega, |x|^{-1}):=\{f\in C_\mathrm{hom}^{2,\alpha}(\Omega, |x|^{-1}): \vec{X}\cdot\nabla f\in C_\mathrm{hom}^{0,\alpha}(\Omega, |x|^{-1})\},
\]
with norms
\[
\|f:C_\mathrm{an}^{2,\alpha}(\Omega, |x|^{-1})\|^2:=\|f:C_\mathrm{hom}^{2,\alpha}(\Omega, |x|^{-1})\|^2+\|\vec{X}\cdot\nabla f:C_\mathrm{hom}^{0,\alpha}(\Omega, |x|^{-1})\|^2.
\]
\end{definition}

The definition of the homogeneously weighted spaces are motivated partly by the following lemma. Note also that $C_\mathrm{hom}^{0,\alpha}(\Omega,|x|^{k})\subseteq e^{|x|^2/8}L^2(\Omega)$.
\begin{lemma}\label{LemmaHomo}
Let $h(x)=c(\frac{x}{|x|})|x|^{k}$ be homogeneous of degree $k\in\mathbb{Z}$, where $c\in C^{2,\alpha}(S^1)$, then
\beq\label{HomoLemma}
\begin{split}
&(\nabla)^l h\in C_\mathrm{hom}^{0,\alpha}(\Omega,|x|^{k-l}),\quad l=0,1,2,\quad\textrm{with}\\
& \|(\nabla h)^l\|_{C_\mathrm{hom}^{0,\alpha}(\Omega,|x|^{k-l})}\leq \|c\|_{C^{l,\alpha}(S^1)}.
\end{split}
\eeq
When $k=1$, then we have the property
\[
\mathcal{L}_\mathcal{P}h\in C_\mathrm{hom}^{0,\alpha}(\Omega,|x|^{-1}).
\]
Furthermore $\mathcal{L}_\mathcal{P}:C_\mathrm{an}^{2,\alpha}(\Omega,|x|^{-1})\to C_\mathrm{hom}^{0,\alpha}(\Omega,|x|^{-1})$ is a bounded operator.
\end{lemma}
\begin{proof}
The first claim for homogeneous functions $h$ is elementary from the definitions, using scaling.

When $k=1$, $\mathcal{L}_\mathcal{P}h=\Delta h-\fracsm{1}{2}(\vec{X}\cdot\nabla-1)h=\Delta h$ is a sum of homogeneous functions, namely one of degree $-1$ and one of degree $-2$, and the second and third result also follow.
\end{proof}

\begin{definition}\label{cone_spaces}
The (anisotropic homogeneous) H\"older cone space of functions asymptotic to graphical cones over the plane, are:
\begin{align}
&\mathcal{CS}^{0,\alpha}(\Omega, |x|^{-1}):=C_\mathrm{hom}^{0,\alpha}(\Omega, |x|^{-1}),\\
&\mathcal{CS}^{2,\alpha}(\Omega, |x|^{-1}):=C^{2,\alpha}(\partial\Omega)\times C_\mathrm{an}^{2,\alpha}(\Omega, |x|^{-1}),
\end{align}
the latter equipped with the product norm
\[
\|(c,f):\mathcal{CS}^{2,\alpha}(\Omega, |x|^{-1})\|^2:=\|c\|^2_{C^{2,\alpha}(S^1)}+\|f:C_\mathrm{an}^{2,\alpha}(\Omega, |x|^{-1})\|^2.
\]

\end{definition}
\begin{remark}
\begin{itemize}
\item[]
\item[(i)] The pairs $(c,f)$ injectively model graphs $u:\Omega\to\Reals$ as follows,
\beq
u=u_{(c,f)}(r,\theta):=c(\theta)r+f(r,\theta),
\eeq
in polar coordinates, and by abuse of notation we write $u=(c,f)$.
\item[(ii)] An important consequence in this context, is that our linearized operator in (\ref{EqPlane}) induces a well-defined bounded map $(c,f)\mapsto \mathcal{L}_{\mathcal{P}}(u_{(c,f)})$,
\beq\label{L_bounded}
\mathcal{L}_{\mathcal{P}}:\mathcal{CS}^{2,\alpha}(\Omega, |x|^{-1})\to \mathcal{CS}^{0,\alpha}(\Omega, |x|^{-1})=C_\mathrm{hom}^{0,\alpha}(\Omega, |x|^{-1}),
\eeq
as opposed to second order operators generally (e.g. $\Delta+1$).
\end{itemize}
\end{remark}

\begin{proposition}\label{Compactness}
The spaces $C_\mathrm{hom}^{k,\alpha}(\Omega, |x|^{-1})$, $C_\mathrm{an}^{k,\alpha'}(\Omega, |x|^{-1})$ and $\mathcal{CS}^{2,\alpha}(\Omega, |x|^{-1})$ are Banach, and the natural inclusions for $0<\alpha<\alpha'<1$,
\begin{align}
&C_\mathrm{hom}^{k,\alpha'}(\Omega, |x|^{-1})\hookrightarrow C_\mathrm{hom}^{k,\alpha}(\Omega, |x|^{-1+l}),\label{Emb_hom}\\
&C_\mathrm{an}^{k,\alpha'}(\Omega, |x|^{-1})\hookrightarrow C_\mathrm{an}^{k,\alpha}(\Omega, |x|^{-1+l}),\label{Emb_an}\\
&\mathcal{CS}^{2,\alpha'}(\Omega, |x|^{-1})\hookrightarrow \mathcal{CS}^{2,\alpha}(\Omega, |x|^{-1+l})\label{Emb_cones},
\end{align}
are compact, where $l>0$ is arbitrary and signifies a weakening of the decay rate (we will here only use $|x|^0$, so $l=1$).
\end{proposition}
\begin{proof}
It is a standard exercise to verify that these spaces are complete with the norms we have defined.

Since $\Omega$ is non-compact, it is for the compactness of the embeddings (\ref{Emb_hom})-(\ref{Emb_cones}) to be true crucial that: (A) We have arranged that the weight functions on all derivatives are decaying, and (B) Cones are modeled by functions on a compact curve in $\mathbb{S}^2$, here on $\partial\Omega=S^{1}$. Note that it is an important special feature of the operator $\mathcal{L}_\mathcal{P}$ that the property (B) can be brought into play (see the Liouville result in Proposition \ref{liouville_type_result}).

Namely, for any bounded domain $D\subset\subset\Reals^n$ the embeddings $C^{k,\alpha'}(D)\hookrightarrow C^{k,\alpha}(D)$, of the usual H\"o{}lder spaces, are compact if $0<\alpha<\alpha'<1$, as follows from the Arzel\`a{}-Ascoli theorem. This fact along with a standard cut-off argument and the property (A) shows that the embeddings in (\ref{Emb_hom}) and (\ref{Emb_an}) are compact.

From the compactness of (\ref{Emb_an}) and the property (B), i.e. compactness of $C^{k,\alpha'}(\partial\Omega)\hookrightarrow C^{k,\alpha}(\partial\Omega)$, it now finally follows that also
\beq
C^{2,\alpha'}(\partial\Omega)\times C_\mathrm{an}^{2,\alpha'}(\Omega, |x|^{-1})\hookrightarrow C^{2,\alpha}(\partial\Omega)\times C_\mathrm{an}^{2,\alpha}(\Omega, |x|^{-1+l})
\eeq
is compact, completing the proof of (\ref{Emb_cones}).
\end{proof}

\subsection{Homogeneously weighted H\"o{}lder estimates}

In this section we prove the second derivative Schauder estimates in the weighted H\"{o}lder spaces. Recall that we take $\Omega=\Reals^2\setminus B_R(0)$ to be a domain exterior to a disk.
\begin{proposition}\label{HessianEstimate}
If $E\in C_\mathrm{hom}^{0,\alpha}(\Omega,|x|^{-1})$ and $v\in e^{|x|^2/8}H^2(\Omega)\cap C_{\mathrm{loc}}^{2,\alpha}(\Omega)$ is a solution to $\mathcal{L}_\mathcal{P}v=\Delta v - \fracsm{1}{2}(\vec{X}\cdot\nabla - 1)v=E$, then
\[
D_{x_ix_j}v\in C_\mathrm{hom}^{0,\alpha}(\Omega, |x|^{-1}),
\]
and if $v_{\mid\partial\Omega}=0$ there is a constant $C=C(\alpha)>0$ such that
\beq\label{Hess}
\|D_{x_ix_j}v\|_{C_\mathrm{hom}^{0,\alpha}(\Omega,|x|^{-1})}\leq C \|E\|_{C_\mathrm{hom}^{0,\alpha}(\Omega,|x|^{-1})}.
\eeq
\end{proposition}

\begin{proof}
There are several routes one may take to prove such a result, for example the resolvents can be found in the form of contour integrals by summing up the eigenfunctions via Mehler's formula.

However, using the well-known connection to parabolic equations (and whence this problem came, of course) is less involved. Namely, the equation
\[
\mathcal{L}_\mathcal{P}u=\Delta u - \fracsm{1}{2}(\vec{X}\cdot\nabla - 1)u = E,
\]
is the elliptic equation describing a backwards self-similar solution to the flat space heat equation, but with a modified source term.

It is convenient to consider a fixed extension map $v\mapsto \tilde{v}\in C_{\mathrm{loc}}^{2,\alpha}(\Reals^2)$ with the property
\beq
\|\tilde{v}\|_{C^{2,\alpha}(B_R(0))}\leq C\|v\|_{C^{2,\alpha}(B_{R+1}(0)\setminus B_{R}(0))},
\eeq
where the constant is independent of $v$. Then letting $\tilde{E}=\mathcal{L}_\mathcal{P}\tilde{v}$ we see that $\tilde{E}\in C_\mathrm{hom}^{0,\alpha}(\Reals^2,|x|^{-1})$ and
\begin{align*}
\|\tilde{E}\|_{C_\mathrm{hom}^{0,\alpha}(\Reals^2,|x|^{-1})}&\leq \|E\|_{C_\mathrm{hom}^{0,\alpha}(\Omega, |x|^{-1})}+C\|\tilde{E}\|_{C^{0,\alpha}(B_R(0))}\\
&\leq \|E\|_{C_\mathrm{hom}^{0,\alpha}(\Omega, |x|^{-1})}+ C\|v\|_{C^{2,\alpha}(B_{R+1}(0)\setminus B_{R}(0))}\\
&\leq \|E\|_{C_\mathrm{hom}^{0,\alpha}(\Omega, |x|^{-1})}+ C\big[\|E\|_{C^{0,\alpha}(B_{R+1}(0)\setminus B_{R}(0))}+ \sup_{x\in\Omega}(1+|x|)|E(x)|\big]\\
&\leq C\|E\|_{C_\mathrm{hom}^{0,\alpha}(\Omega, |x|^{-1})},
\end{align*}
where in the second to last estimate we used Schauder estimates (such as Theorem 10.2.1-10.2.2 in \cite{Jo}), using the fact that $v_{\mid\partial B_R}=0$ and the bounds on $v$ from the second part of Lemma \ref{LinearGrowth}. Now, since also automatically
\beq
\|D_{x_ix_j}v\|_{C_\mathrm{hom}^{0,\alpha}(\Omega, |x|^{-1})}\leq \|D_{x_ix_j}\tilde{v}\|_{C_\mathrm{hom}^{0,\alpha}(\Reals^2, |x|^{-1})},
\eeq
we see that it is enough to prove the estimate (\ref{Hess}) for $\tilde{v}$ and $\tilde{E}$, so we assume without loss of generality that $v$ and $E$ are defined on $\Reals^2$.

The elliptic equation is now, as mentioned above, easily rewritten to the condition
\beq\label{VDef}
v(x,t):=\sqrt{1-t}\:u\big(\fracsm{x}{\sqrt{1-t}}\big)
\eeq
solves the following heat equation
\beq
\begin{cases}
&\partial_tv-\Delta v=F(x,t),\quad (t,x)\in (0,1)\times\Reals^2,\\
&v(x,0)=u(x), \quad x\in\Reals^2,\\
\end{cases}
\eeq
where the correspondingly transformed source term now reads:
\beq
F(x,t):=-\frac{E\left(\fracsm{x}{\sqrt{1-t}}\right)}{\sqrt{1-t}}.
\eeq

Now, recall the heat kernel in Euclidean space,
\[
\Phi(x-y,t-s):=\frac{1}{4\pi (t-s)}e^{-\frac{|x-y|^2}{4(t-s)}}.
\]

Note that we have the following representation formula which allows us to use standard methods of proof (e.g. the standard, non-weighted Schauder theory for the heat equation. See for example \cite{La})
\beq\label{Representation}
\begin{split}
D_{x_ix_j}v(x,t)=&\int_{\Reals^2}D_{x_ix_j}\Phi(x-y,t)u(y)dy\\
&-\int_0^t\int_{\Reals^2}D_{x_ix_j}\Phi(x-y,t-s)\Big[F(x,s)-F(y,s)\Big]dyds,
\end{split}
\eeq
where
\beq\label{HeatKernelDeriv}
D_{x_ix_j}\Phi(x-y,t-s)=\Big[\frac{(x_i-y_i)(x_j-y_j)}{4(t-s)^2}-\frac{\delta_{ij}}{2(t-s)}\Big]\Phi(x-y,t-s),
\eeq
and we have subtracted a term which is zero. The expression is well-defined when $F$ is H\"older in the $x$-variable, and justified by inserting a cut-off $\chi_h(t)$, supported away from $t=1$,  then differentiating under the integral and finally letting $h\to 0$.

Note from (\ref{HeatKernelDeriv}) the useful inequality (for constants $A,B>0$):
\beq\label{KernelEstimates}
|D_{x_ix_j}\Phi(x-y,t-s)|\leq A(t-s)^{-2}e^{-B\frac{|x-y|^2}{t-s}}
\eeq
and note also that since $E\in C_\mathrm{hom}^{0,\alpha}(|x|^{-1})$
\[
\left|F(x,s)-F(y,s)\right|\leq \frac{\|E\|_{C_\mathrm{hom}^{0,\alpha}(|x|^{-1})}}{2}|x-y|^{\alpha}\big(|x|^{-1-\alpha}+|y|^{-1-\alpha}\big),
\]
by the way we have defined $C_\mathrm{hom}^{0,\alpha}(|x|^{-1})$.

Let us first prove that with $E\in C_\mathrm{hom}^{0,\alpha}(|x|^{-1})$, we have
\beq\label{SupOfHess}
\sup_{x\in\Reals^2\setminus B_R(0)}(1+|x|)|D_{x_ix_j}u(x)|\leq C\|E\|_{C_\mathrm{hom}^{0,\alpha}(|x|^{-1})}
\eeq

Note that by virtue of the scaling in the definition of $v$, it suffices for (\ref{SupOfHess}) to establish that
\[
\sup_{t\in(t_R,1)}\sup_{|x|=1}|D_{x_ix_j}v(x,t)|\leq C\|E\|_{C_\mathrm{hom}^{0,\alpha}(|x|^{-1})},
\]
where $t_R:=1-\frac{1}{2R^2}$. Let us fix $R=2$, such that $t_R=\frac{7}{8}$. We see that for $|x|=1$ we have from Equation (\ref{Representation})
\begin{align*}
&|D_{x_ix_j}v(x,t)|\leq C\|E\|_{C_\mathrm{hom}^{0,\alpha}(|x|^{-1})}\int_{\Reals^2}e^{-B'|y|^2}(1+|y|)dy\\
&\quad\quad+ C'\|E\|_{C_\mathrm{hom}^{0,\alpha}(|x|^{-1})}\int_0^1\int_{\Reals^2}|x-y|^{\alpha}\big(|x|^{-1-\alpha}+|y|^{-1-\alpha}\big)(1-s)^{-2}e^{-B\frac{|x-y|^2}{1-s}}dyds\\
&\quad\quad = C''\|E\|_{C_\mathrm{hom}^{0,\alpha}(|x|^{-1})},
\end{align*}
where we used $|u(x-y)|\leq C\|E\|_{C_\mathrm{hom}^{0,\alpha}(|x|^{-1})}(1+|y|)$ as well as (\ref{KernelEstimates}) to estimate the first term, and where of course the integral
\[
\int_0^1\int_{\Reals^2}|y|^{\alpha}(1-s)^{-2}e^{-B\frac{|y|^2}{1-s}}dyds<\infty,\quad\mathrm{for}\:\mathrm{any}\:\alpha>0.
\]

Again, by the scaling in (\ref{VDef}) and our definition of the weighted spaces, the desired estimate for the H\"older coefficients will follow if we can show that
\[
\sup_{t\in(t_R,0)}\sup_{|x_0|\leq|x|=1}|D_{x_ix_j}v(x,t)-D_{x_ix_j}v(x_0,t)|\leq C\|E\|_{C_\mathrm{hom}^{0,\alpha}(|x|^{-1})}|x-x_0|^\alpha.
\]

Hence we compute for $|x_0|\leq|x|\leq 1$:
\beq
\begin{split}
&D_{x_ix_j}v(x,t)-D_{x_ix_j}v(x^0,t)=\\
&\quad\int_{\Reals^2}\Big(D_{x_ix_j}\Phi(x-y,t)-D_{x^0_ix^0_j}\Phi(x^0-y,t)\Big)u(\fracsm{y}{\sqrt{2}})dy\\
&\quad+\int_0^t\int_{|x-y|\leq2|x-x^0|}D_{x_ix_j}\Phi(x-y,t-s)\Big[F(y,s)-F(x,s)\Big]dyds\\
&\quad -\int_0^t\int_{|x-y|\leq2|x-x^0|}D_{x^0_ix^0_j}\Phi(x^0-y,t-s)\Big[F(y,s)-F(x^0,s)\Big]dyds\\
&\quad +\int_0^t\int_{|x-y|\geq2|x-x^0|}\Big(D_{x_ix_j}\Phi(x-y,t-s)-D_{x^0_ix^0_j}\Phi(x^0-y,t-s)\Big)\Big[F(y,s)-F(x,s)\Big]dyds\\
&\quad -\int_0^t\int_{|x-y|\geq2|x-x^0|}D_{x^0_ix^0_j}\Phi(x^0-y,t-s)\Big[F(x,s)-F(x^0,s)\Big]dyds\\
&\quad=:I_1+\ldots+I_5.
\end{split}
\eeq
In this expression, the first term is estimated using the mean value principle, such that for
\[
|D_{x_ix_j}\Phi(x-y,t)-D_{x^0_ix^0_j}\Phi(x^0-y,t)|\leq |y+\xi|\frac{|x-x^0|}{t}\Phi(\xi+y,t)\leq C' e^{-B'\frac{|y+\xi|^2}{t}} |x-x^0|^\alpha,
\]
for some point $\xi$ on the line between the points $x^0$ and $x$, so $|\xi|\leq 2$, and some constants $B',C'=C(t_R)$ independent of $|x|,|x^0|\leq 1$. Hence one gets the estimate
\beq
|I_1|\leq C|x-x^0|^\alpha\|E\|_{C_\mathrm{hom}^{0,\alpha}(|x|^{-1})}\int_{\Reals^2}e^{-B'\frac{|y|^2}{t}}(|y|+2)dy.
\eeq

The terms $I_2$ and $I_3$ are of course symmetric in $x\leftrightarrow x^0$ and have similar estimates. For $I_2$ we get:
\beq
\begin{split}
|I_2|&\leq C\|E\|_{C_\mathrm{hom}^{0,\alpha}(|x|^{-1})}\int_0^t\int_{|x-y|\leq2|x-x^0|}(t-s)^{-2}e^{-B\frac{|x-y|^2}{t-s}}|x-y|^\alpha dyds\\
&\leq C\|E\|_{C_\mathrm{hom}^{0,\alpha}(|x|^{-1})}\int_{|x-y|\leq2|x-x^0|}|x-y|^{-2+\alpha}\\
&\leq C\|E\|_{C_\mathrm{hom}^{0,\alpha}(|x|^{-1})}|x-x^0|^\alpha.
\end{split}
\eeq

For the term $I_4$ we use the estimate
\[
|D_{x_ix_j}\Phi(x-y,t)-D_{x^0_ix^0_j}\Phi(x^0-y,t)|\leq c|x-x^0|(t-s)^{-5/2}e^{-B\frac{|x-y|^2}{t-s}},
\]
which holds whenever $|x-y|\geq 2|x-x^0|$. Hence we see
\beq
\begin{split}
|I_4|&\leq C\|E\|_{C_\mathrm{hom}^{0,\alpha}(|x|^{-1})}|x-x^0|\int_0^t\int_{|x-y|\geq2|x-x^0|}(t-s)^{-5/2}e^{-B\frac{|x-y|^2}{t-s}}|x-y|^\alpha dyds\\
&\leq C \|E\|_{C_\mathrm{hom}^{0,\alpha}(|x|^{-1})}|x-x^0|^\alpha.
\end{split}
\eeq

For the last term, we rewrite it as
\[
I_5=-\int_0^t\int_{|x-y|=2|x-x^0|}\frac{\partial\Phi(x^0-y,t-s)}{\partial x^0_j}\Big[F(x,s)-F(x^0,s)\Big](\vec{e}_i\cdot\vec{\nu})dM(y)ds,
\]
where $\vec{e}_i$ the $i$th unit vector in $\Reals^2$, $dM(y)$ is the line element and $\vec{\nu}$ the outward pointing unit normal to the disk of radius $2|x-x^0|$. Since
\beq
\frac{\partial\Phi(x^0-y,t-s)}{\partial x^0_j}=-\frac{x^0_j-y_j}{8\pi(t-s)^2}e^{-\frac{|x^0-y|^2}{4(t-s)}},
\eeq
we finally get
\[
\begin{split}
|I_5|&\leq \|E\|_{C_\mathrm{hom}^{0,\alpha}(|x|^{-1})}|x-x^0|^{1+\alpha}\int_0^t\int_{|x-y|=2|x-x^0|}\frac{e^{\frac{-|x-x^0|^2}{4(t-s)}}}{4\pi(t-s)^2}dM(y)ds\\
& = C\|E\|_{C_\mathrm{hom}^{0,\alpha}(|x|^{-1})}|x-x^0|^{\alpha}.
\end{split}
\]
\end{proof}

Using the second derivative bound we can now proceed to our final proposition of this section, which is a Liouville-type structure theorem in that we prove solutions are homogeneous degree one polynomials in $x$ plus a remainder belonging to the space $C_\mathrm{an}^{2,\alpha}(|x|^{-1})$. This detailed analysis of the solutions -- completing our separation of the conical part -- is exactly what will make our construction work.

\begin{theorem}[Liouville-type result] \label{liouville_type_result}
There is a constant $C>0$ s.t. for any $G_\tau$-equivariant $E\in C_\mathrm{hom}^{0,\alpha}(\Omega, |x|^{-1})$ there exists a unique $G_\tau$-equivariant $u=(c,f)\in\mathcal{CS}^{2,\alpha}(\Omega, |x|^{-1})$ such that
\[
\mathcal{L}_{\mathcal{P}}\big[c(\theta)\cdot |x| +f(x)\big]=\mathcal{L}_{\mathcal{P}}u=E,
\]
where $u=u_{(c,f)}$ and $u=0$ on $\partial\Omega$. Furthermore we have the estimate
\beq\label{ExteriorEstimate}
\|(c,f):\mathcal{CS}^{2,\alpha}(\Omega, |x|^{-1})\|\leq C \|E:C_\mathrm{hom}^{0,\alpha}(\Omega, |x|^{-1})\|.
\eeq
\end{theorem}
\begin{proof}
Let $u\in C^{2,\alpha}_\mathrm{loc}(\Omega)$ be a solution to $\mathcal{L}u=E$. It follows from the weighted H\"o{}lder estimates in Proposition \ref{HessianEstimate} that
\beq\label{Hess_decay}
D_{x_ix_j}u\in C_\mathrm{hom}^{0,\alpha}(\Omega, |x|^{-1}),
\eeq
and hence we have $\Delta u\in C_\mathrm{hom}^{0,\alpha}(\Omega, |x|^{-1})$ and hence $w:=-\vec{X}\cdot\nabla u+u=E-\Delta u\in C_\mathrm{hom}^{0,\alpha}(\Omega, |x|^{-1})$. Solving for $u$ in polar coordinates ($|x|=r$), we get after imposing initial conditions $u_{|\partial\Omega}=0$ that (normalize here for simplicity the radius $R$ of $\partial\Omega$ to 1):
\begin{align}
&u(r,\theta)=c(\theta)r+v_0,\label{u_decomp}\\
&c(\theta):=-\int_1^\infty\frac{w(s,\theta)}{s^2}ds=\lim_{r\to\infty}\frac{u(r,\theta)}{r},\\
&v_0(x):=r\int_{r}^\infty\frac{w(s,\theta)}{s^2}ds. \label{v_0}
\end{align}
By (\ref{Hess_decay}) and (\ref{u_decomp}), and the lemma for homogeneous functions (\ref{HomoLemma}), we see that also $D_{x_ix_j}v_0\in C_\mathrm{hom}^{0,\alpha}(\Omega, |x|^{-1})$.

Now, $-\vec{X}\cdot\nabla v_0+v_0=-\vec{X}\cdot\nabla u+u=w\in C_\mathrm{hom}^{0,\alpha}(\Omega, |x|^{-1})$ from above. It follows easily from the formula (\ref{v_0}) for $v_0$ that $v_0\in C_\mathrm{hom}^{0,\alpha}(\Omega, |x|^{-1})$, and hence we see that also $\vec{X}\cdot\nabla v_0\in C_\mathrm{hom}^{0,\alpha}(\Omega, |x|^{-1})$.

It remains to show that the full gradient satisfies
\beq\label{Interpol}
\nabla v_0\in C_\mathrm{hom}^{0,\alpha}(\Omega, |x|^{-1}),
\eeq
Note for this, that
\beq
\Delta v_0 = E-\Delta (c(\theta)|x|)+\vec{X}\cdot\nabla v_0-v_0\in C_\mathrm{hom}^{0,\alpha}(\Omega, |x|^{-1}).
\eeq

Equation (\ref{Interpol}) follows now easily from this with $v_0\in C_\mathrm{hom}^{0,\alpha}(\Omega, |x|^{-1})$, by standard use of the Green's function for the ordinary flat Laplacian on the plane (see for example the estimate (10.1.30) in \cite{Jo}).

Hence we have shown that there is the desired Liouville decomposition, and the corresponding estimates follow.
\end{proof}

\section{Linearized equation on the initial surface $\mathcal{M}[\tau,\theta]$}

We let $\underline{a}:=8|\log\tau|$ and then $\mathcal{N}^\pm_y, \mathcal{N}^\pm_x$ are used to denote the connected components of $\{\mathcal{M}[\tau,\theta]: s\geq \underline{a}\}$. Let also $\mathcal{S}:=\mathcal{H}(\Sigma)$, where we denote by $\mathcal{H}$ the homothety by a factor of $\tau$.

\begin{definition}\label{NormsOnInitialSurface}
 Let $v\in C_\mathrm{loc}^{k,\alpha}(\mathcal{M})$. We identify $v$ with its restrictions to $\Sigma$, $\mathcal{N}^\pm_y$ and $\mathcal{N}^\pm_x$.  Then for $k= 0, 2$ we define the norm $\|v\|_{\mathcal{XS}^{k,\alpha}}$ to be the maximum of the following quantities, where $b_0=e^{-5\delta_s/\tau}$ and $b_2=e^{-5\delta_s/\tau}/\tau^{10}$:
\begin{itemize}
\item[(1)] $\tau^{1-k}\|v\circ \mathcal{H}\|_{C^{k, \alpha}(\Sigma, e^{- \gamma s}, g_\Sigma)}$, and
\item[(2)] $b_k^{-1}\|v\|_{\mathcal{CS}^{k,\alpha}(\mathcal{N}^+_x\setminus\mathcal{S}, |x|^{-1})}$, as given in Definition \ref{cone_spaces}.
\item[ (3)] $b_k^{-1} \|v\|_{C^{k, \alpha} (\mathcal{N}^\pm_y\setminus\mathcal{S},g_{\mathcal{N}^\pm_y})}$, and $b_k^{-1}\|v\|_{C^{k, \alpha} (\mathcal{N}_x^-\setminus\mathcal{S},g_{\mathcal{N}^\pm_y})}$.
\end{itemize}
We let be $\mathcal{XS}^{k, \alpha}(\mathcal{M})$ be the space of functions $v$ for which $\|v\|_{\mathcal{XS}^{k,\alpha}}<\infty$.
\end{definition}

\begin{lemma} \label{invertibility_on_ends}
Let $\mathcal{N}_i$ stand for any of the ends $\mathcal{N}^\pm_y$, $\mathcal{N}^\pm_x$. Then for $\tau>0$ sufficiently small the Dirichlet operator, for zero initial value on $\partial\mathcal{N}_i$,
\begin{equation} \notag
\mathcal{L}_{\mathcal{N}_i}: \mathcal{XS}_0^{2,\alpha}(\mathcal{N}_i) \rightarrow \mathcal{XS}^{0,\alpha}(\mathcal{N}_i)
\end{equation}
is invertible, with operator norm of the inverse bounded uniformly in  $\tau>0$.
\end{lemma}
\begin{proof}
For the exterior flat domain, this is what was proved in Section 8. For the flat disk and round spherical cap, we check the invertibility by computing the Dirichlet spectrum of the stability operator $\mathcal{L}$ on these surfaces, using a perturbation argument to extend the property to the $\theta$-family of spherical caps (by possibly taking $\delta_\theta$ smaller). These spectrum computations can be found in the Appendix.

Note however that we are considering the region of $\tau\Sigma[\tau,\theta]$, very near the removed circle, and here the initial surface $\mathcal{M}[\tau,\theta]$ and hence the ends $\mathcal{N}$, do not exactly coincide with the subsets of the configuration $\mathcal{C}[\theta]$. The difference is on each piece a small normal graph with compact support, coming from the function $f(s,z)$ describing Scherk's surface as a graph over its four asymptotic planes. But by construction and the estimates (\ref{f_estimate}) we verify that the cut-off $\underline{a}=8\log\tau$ is appropriately large, since for the two induced metrics in question,
\[
\|g_{\mathcal{N}_i}-g_{\mathcal{C}[\theta]}:C^3(\{\tau\Sigma[\tau,\theta]:s\geq\underline{a}\},g_{\mathcal{C}[\theta]})\|\leq C\tau^{-2}e^{-\underline{a}}=C\tau^6,
\]
and similarly for the induced second fundamental forms $|A|^2$, and hence the lemma follows for small enough $\tau>0$ by a perturbation within the compact domain $\{\tau\Sigma[\tau,\theta]:s\geq\underline{a}\}$, for the quantities used in the definition of $\mathcal{L}$.
\end{proof}

Note that the property (\ref{L_bounded}) extends so that also $\mathcal{L}_\mathcal{M}$, the linearized operator of $H-\langle\vec{X},\vec{\nu}\rangle$ over $\mathcal{M}$, is a bounded map from the H\"o{}lder cone space.

\begin{definition}
Let $\Theta: [-\delta_\theta,\delta_\theta] \to C^{\infty} (\mathcal{M})$ be given by
	\[
	\Theta(\theta) =\frac{1}{\tau}\mathcal{H}^*(\theta w),
	\]
where $\theta \in [-\delta_\theta,\delta_\theta]$.
\end{definition}

\begin{theorem}\label{LinearOnAll}
Given $E \in \mathcal{XS}^{0,\alpha}(\mathcal{M})$, there exist  $v_E \in \mathcal{XS}^{2,\alpha}(\mathcal{M})$ and $\theta_E \in \Reals$, such that
	\[
	\mathcal{L}_{\mathcal{M}}v_E=E+\Theta(\theta_E),
	\]
and 
	\[
	\| v_E \|_{\mathcal{XS}^{2,\alpha}} \leq C \|E\|_{\mathcal{XS}^{0,\alpha}}, |\theta_E |\leq C \|E\|_{\mathcal{XS}^{0,\alpha}}, 
	\]
\end{theorem}

\begin{proof}

Let the cut-off functions $\psi:= \psi[5\delta_s /\tau, 5\delta_s /\tau -1]\circ s$ as well as $\psi':= \psi[\underline a ,\underline a+1 ]\circ s$ be given on $\mathcal{M}$, and let $\underline a=8 |\log \tau|$. 

The starting point of our iteration is $E_0:=E$. Applying Proposition \ref{invertibility} to $\Sigma=\Sigma[\tau,\theta]=\mathcal{H}^{-1}(\mathcal{S})$ with the cut-off source term $E':=\tau(\psi E_{n-1})\circ\mathcal{H}$. From the corresponding $v_E$ we get $v:=\tau\mathcal{H}_\ast(v_E)$ and we let the $\theta_n:=\theta_{E'}$. By construction we have thus on $\mathcal{S}$ that
\[
\mathcal{L}_{\mathcal{M}}v=\psi E_{n-1}+\Theta(\theta_n).
\]

We now feed the new source term $E''=(1-\psi^2)E_{n-1}-[\mathcal{L}_{\mathcal{M}},\psi]v$ into the equation on the union of the ends $\mathcal{N}_{\cup}:=\mathcal{N}_y\cup\mathcal{N}^-_x\cup\mathcal{N}^+_x$ (here the commutator is by definition $[\mathcal{L}_{\mathcal{M}}, \psi] f:= \mathcal{L}_{\mathcal{M}} (\psi) f - \psi(\mathcal{L}_{\mathcal{M}} f)$), and obtain a solution $v_{E''}$ which we call $v'$,
\[
\mathcal{L}_{\mathcal{M}}v'=(1-\psi^2)E_{n-1}-[\mathcal{L}_{\mathcal{M}},\psi]v.
\]

We then finally define
\[
v_n:=\psi v+\psi'v'.
\]
By considering the supports of $\psi,\psi'$ and $[\mathcal{L}_{\mathcal{M}}, \psi]$, we see that
\beq\label{telescope}
	\mathcal{L}_{\mathcal{M}} v_n = E_{n-1}+[\mathcal{L}_{\mathcal{M}}, \psi'] v'+\Theta(\theta_n).
\eeq
We then also define the new source term $E_n=- [\mathcal{L}_{\mathcal{M}}, \psi'] v'$. Again, we use the fact that $[\mathcal{L}_{\mathcal{M}},\psi']$ is supported on $[\underline a, \underline a+1]$, use Lemma \ref{invertibility_on_ends}, and estimate (for $\tau$ sufficiently small),
\begin{align*}
 \| E_n\|_{\mathcal{XS}^{0,\alpha}} &= \tau\|E_n\circ\mathcal{H}:C^{0,\alpha}(\Sigma,g_\Sigma, e^{-\gamma s})\|\\
 &\leq\tau e^{\gamma(\underline a+1)}\|[\mathcal{L}_{\mathcal{M}},\psi']v'\circ\mathcal{H}:C^{0,\alpha}(\Sigma_{[\underline a, \underline a+1]},g_\Sigma)\|\\
 &\leq C \tau^{-p_0} e^{\gamma(\underline a+1)}\|v'\circ\mathcal{H}:C^{2,\alpha}(\Sigma_{[\underline a, \underline a+1]},g_\Sigma)\|\\
 &\leq C \tau^{-p_0'}e^{\gamma(\underline a+1)}e^{-\left(\frac{5\delta_s}{\tau}-1\right)}\|(1-\psi^2)E_{n-1}-[\mathcal{L}_{\mathcal{M}},\psi]v\|_{\mathcal{XS}^{0,\alpha}},
\end{align*}
where we used in the third line the uniform control of the geometry of $\Sigma$ in the strips $s\in[\underline{a},\underline{a}+1]$, and in the third line the Definition \ref{NormsOnInitialSurface}, and the fact that the term considered in the last line has support in $s\in [\frac{5\delta_s}{\tau}-1,\frac{5\delta_s}{\tau}]$, we thus get
\begin{align*}
 \| E_n\|_{\mathcal{XS}^{0,\alpha}} &\leq C\tau^{-p_0'}e^{\gamma(\underline a+1)}e^{-\left(\frac{5\delta_s}{\tau}-1\right)}\|E_{n-1}\|_{\mathcal{XS}^{0,\alpha}}\\
 &\leq C e^{-\frac{\delta_s}{\tau}}\|E_{n-1}\|_{\mathcal{XS}^{0,\alpha}}.
\end{align*}

We define $v_E := \sum_{n=1}^{\infty} v_n$ and $\theta_{E} := \sum_{n=1}^{\infty} \theta_n$. The first sum converges in the Banach space $\mathcal{XS}^{2,\alpha}(\mathcal{M})$, the second converges to some real number which is the $\theta_E$, with the desired estimates. The function  $v_E$ then satisfies $\mathcal{L}_{\mathcal{M}} v_E  =E + \Theta(\theta_E)$.
\end{proof}

\begin{definition} \label{complicated_term_shorthand}
Let $\mathcal{S}$ be a smooth surface (possibly with boundary). For a function $v\in C^{2, \alpha}(\mathcal{S})$ for which $\mathcal{S}_v$ is a $C^{2, \alpha}$-surface, we define on $\mathcal{S}$:
\[
\mathcal{F}_{\mathcal{S}}(v):=H_{\mathcal{S}_v}-\fracsm{1}{2}\langle\vec{X},\vec{\nu}_{\mathcal{S}_v}\rangle,
\]
and denote $\mathcal{F}_{\mathcal{S}}:=\mathcal{F}_{\mathcal{S}}(0)$.
\end{definition}

\begin{corollary}\label{CorLinear}
There are $v_\mathcal{F} \in \mathcal{XS}^{2,\alpha}(\mathcal{M})$ and $\theta_\mathcal{F}$ such that 
	\begin{gather*}
	\mathcal{L}_\mathcal{M}v_\mathcal{F} = \mathcal{F}_\mathcal{M} +\Theta(\theta_\mathcal{F}) w,\\
	|\theta_\mathcal{F} - \theta| \leq C \tau, \quad \| v_\mathcal{F}\|_{\mathcal{XS}^{2,\alpha}} \leq C \tau,
	\end{gather*}
where $\mathcal{M} = \mathcal{M}[\tau, \theta]$.
\end{corollary}

\section{The nonlinear terms in H\"older cone spaces}

\begin{proposition}\label{QuadraticImprovement}
Given $v\in\mathcal{XS}^{2,\alpha}(\mathcal{M})$ with $\|v\|_{\mathcal{XS}^{2,\alpha}}$ smaller than a suitable constant, we have that the graph $\mathcal{M}_v$ over $\mathcal{M}$, is a smooth immersion and moreover
\[
\mathcal{F}_\mathcal{M}(v)-\mathcal{F}_\mathcal{M}-\mathcal{L}_\mathcal{M} v\in \mathcal{XS}^{0,\alpha}(\mathcal{M}),
\]
with the quadratic improvement bounds:
\beq
\|\mathcal{F}_\mathcal{M}(v)-\mathcal{F}_\mathcal{M}-\mathcal{L}_\mathcal{M} v\|_{\mathcal{XS}^{0,\alpha}}\leq C\|v\|_{\mathcal{XS}^{2,\alpha}}^2.
\eeq
\end{proposition}
\begin{proof}
We first deal with the argument needed on the exterior plane $\Omega=\Reals^2\setminus B_R(0)$. Note that $\vec{\nu}\equiv\vec{e_3}$ and the terms for the equation (\ref{SSEq}) read
\beq
\mathcal{F}_\Omega(v)=-\frac{\Hess v(\nabla v,\nabla v)}{(1+|\nabla v|^2)^{3/2}}+\frac{\mathcal{L}_{\mathcal{P}}v}{\sqrt{1+|\nabla v|^2}},
\eeq
where $\mathcal{L}_{\mathcal{P}}$ is again the linearized operator from (\ref{EqPlane}).

Thus we see that for the exterior plane $\Omega$ we have
\begin{align*}
\mathcal{F}_\Omega(v)-\mathcal{F}_\Omega-\mathcal{L}_\Omega v&=-\frac{\Hess v(\nabla v,\nabla v)}{(1+|\nabla v|^2)^{3/2}}+\left[(1+|\nabla v|^2)^{-1/2}-1\right]\mathcal{L}_{\mathcal{P}}v\\
&=:T_1+T_2.
\end{align*}

Let us first estimate the weighted sup-norm. By the Bernoulli inequalities we have the quadratic bounds:
\beq\label{Elementary}
\begin{split}
&\left|(1+|\nabla v|^2)^{-1/2}-1\right|\leq\fracsm{1}{2}|\nabla v|^2,\\
&\left|(1+|\nabla v|^2)^{-3/2}-1\right|\leq\fracsm{3}{2}|\nabla v|^2.
\end{split}
\eeq

We can now estimate the supremum part of the weighted norms:
\begin{align*}
|x|\left|(\mathcal{F}_\Omega(v)-\mathcal{F}_\Omega (0)-\mathcal{L}_\Omega v)(x)\right|&\leq|x||\Hess v||\nabla v|^2+\fracsm{1}{2}|x||\nabla v|^2|\mathcal{L}_{\mathcal{P}}v|\\
&\leq \|v\|_{\mathcal{CS}^{2,\alpha}}^3+\fracsm{1}{2}\|v\|_{\mathcal{CS}^{2,\alpha}}^3= \fracsm{3}{2}\|v\|_{\mathcal{CS}^{2,\alpha}}^3,
\end{align*}
on the exterior of the disk, where we used again the crucial mapping property (\ref{L_bounded}) on the H\"o{}lder cone spaces.

Similar but slightly more involved computations now show that the H\"o{}lder coefficients in the norm are also estimated as claimed. For example it follows by (\ref{Elementary}) that
\begin{align*}
\frac{\left|(1+|\nabla v(x)|^2)^{-1/2}-(1+|\nabla v(y)|^2)^{-1/2}\right|}{|x-y|^\alpha}&\leq\frac{|\nabla v(x)|+|\nabla v(y)|}{2(1+|\nabla v(y)|^2)^{3/2}}\frac{\big|\nabla v(x)-\nabla v(y)\big|}{|x-y|^\alpha}\\
&\leq \left[|x|^{-\alpha}+|y|^{-\alpha}\right]\|v\|^2_{\mathcal{CS}^{2,\alpha}(\Omega,|x|^{-1})}.
\end{align*}
It follows easily that
\beq
T_2:=\left[\frac{1}{\sqrt{1+|\nabla v|^2}}-1\right]\mathcal{L}_{\mathcal{P}}v\in C_\mathrm{hom}^{0,\alpha}(\Omega,|x|^{-1}),
\eeq
since by assumption $\mathcal{L}_{\mathcal{P}}v\in C_\mathrm{hom}^{0,\alpha}(\Omega,|x|^{-1})$ and also $\nabla v\in C_\mathrm{hom}^{0,\alpha}(\Omega,|x|^{0})$. We get the corresponding higher order bounds as follows. Assume without loss of generality that $|x|\geq|y|$, which is reflected in how we distribute terms, and recall the estimates (\ref{HomoLemma}) in Lemma \ref{LemmaHomo}:
\begin{align*}
\frac{|T_2(x)-T_2(y)|}{|x-y|^\alpha}&\leq \frac{\left|(1+|\nabla v(x)|^2)^{-1/2}-(1+|\nabla v(y)|^2)^{-1/2}\right|}{|x-y|^\alpha}|(\mathcal{L}_\mathcal{P}v)(x)|\\
&\quad +\fracsm{1}{2}|\nabla v|^2\frac{|(\mathcal{L}_\mathcal{P}v)(x)-(\mathcal{L}_\mathcal{P}v)(y)|}{|x-y|^\alpha}\\
&\leq \left[|x|^{-\alpha}+|y|^{-\alpha}\right]\|v\|^2_{\mathcal{CS}^{2,\alpha}(\Omega,|x|^{-1})}\frac{\|\mathcal{L}_\mathcal{P}v\|_{C_\mathrm{hom}^{0,\alpha}(\Omega,|x|^{-1})}}{|x|}\\
&\quad+\fracsm{1}{2}\left[|x|^{-1-\alpha}+|y|^{-1-\alpha}\right]\|v\|^2_{\mathcal{CS}^{2,\alpha}(\Omega,|x|^{-1})}\|\mathcal{L}_\mathcal{P}v\|_{C_\mathrm{hom}^{0,\alpha}(\Omega,|x|^{-1})}\\
&\leq \frac{3}{2}\left[|x|^{-1-\alpha}+|y|^{-1-\alpha}\right]\|v\|^3_{\mathcal{CS}^{2,\alpha}(\Omega,|x|^{-1})}.
\end{align*}
As for the term $T_1$, we write $\Hess v(\nabla v,\nabla v)=\sum_{i,j}(D_{x_ix_j}v)(D_{x_i} v)(D_{x_j} v)$. We again have an estimate
\begin{align*}
\frac{\left|(1+|\nabla v(x)|^2)^{-3/2}-(1+|\nabla v(y)|^2)^{-3/2}\right|}{|x-y|^\alpha}&\leq\frac{3}{2}\frac{|\nabla v(x)|+|\nabla v(y)|}{(1+|\nabla v(y)|^2)^{5/2}}\frac{\big|\nabla v(x)-\nabla v(y)\big|}{|x-y|^\alpha}\\
&\leq 3\left[|x|^{-\alpha}+|y|^{-\alpha}\right]\|v\|^2_{\mathcal{CS}^{2,\alpha}(\Omega,|x|^{-1})}.
\end{align*}

We find by the above, since we may again assume $|x|\geq |y|$,
\begin{align*}
\frac{|T_1(x)-T_1(y)|}{|x-y|^\alpha}&\leq \frac{|(D_{x_ix_j}v)(x)-(D_{x_ix_j}v)(y)|}{|x-y|^{\alpha}}\frac{|\nabla v(y)|^2}{(1+|\nabla v(y)|^2)^{3/2}}\\
&\quad+|(D_{x_ix_j}v)(x)|\frac{|\nabla v(x)-\nabla v(y)|}{|x-y|^\alpha}\frac{|\nabla v(y)|}{(1+|\nabla v(y)|^2)^{3/2}}\\
&\quad+|(D_{x_ix_j}v)(x)|\frac{|\nabla v(x)-\nabla v(y)|}{|x-y|^\alpha}\frac{|\nabla v(x)|}{(1+|\nabla v(y)|^2)^{3/2}}\\
&\quad+|(D_{x_ix_j}v)(x)||\nabla v(x)|^2\frac{\left|(1+|\nabla v(x)|^2)^{-3/2}-(1+|\nabla v(y)|^2)^{-3/2}\right|}{|x-y|^\alpha}\\
&\leq \left[|x|^{-1-\alpha}+|y|^{-1-\alpha}\right]\|v\|_{\mathcal{CS}^{2,\alpha}(\Omega,|x|^{-1})}\|v\|^2_{\mathcal{CS}^{2,\alpha}(\Omega,|x|^{-1})}\\
&\quad+2|x|^{-1}\|v\|_{\mathcal{CS}^{2,\alpha}(\Omega,|x|^{-1})}\left[|x|^{-\alpha}+|y|^{-\alpha}\right]\|v\|_{\mathcal{CS}^{2,\alpha}(\Omega,|x|^{-1})}\|v\|_{\mathcal{CS}^{2,\alpha}(\Omega,|x|^{-1})}\\
&\quad+3|x|^{-1}\|v\|_{\mathcal{CS}^{2,\alpha}(\Omega,|x|^{-1})}\|v\|^2_{\mathcal{CS}^{2,\alpha}(\Omega,|x|^{-1})}\left[|x|^{-\alpha}+|y|^{-\alpha}\right]\|v\|^2_{\mathcal{CS}^{2,\alpha}(\Omega,|x|^{-1})}\\
&\leq 3\left[|x|^{-1-\alpha}+|y|^{-1-\alpha}\right]\left(\|v\|^3_{\mathcal{CS}^{2,\alpha}(\Omega,|x|^{-1})}+\|v\|^5_{\mathcal{CS}^{2,\alpha}(\Omega,|x|^{-1})}\right).
\end{align*}
Collecting these estimates, we have shown:
\[
\|\mathcal{F}_\Omega(v)-\mathcal{F}_\Omega(0)-\mathcal{L}_\Omega v:C_\mathrm{hom}^{0,\alpha}(\Omega,|x|^{-1})\|\leq C\big(\|v\|^3_{\mathcal{CS}^{2,\alpha}(\Omega,|x|^{-1})}+\|v\|^5_{\mathcal{CS}^{2,\alpha}(\Omega,|x|^{-1})}\big).
\]
Picking now $\tau>0$ small enough in terms of $\delta_s$ to ensure $b_0>1$ (and hence also $b_2>1$) in the Definition \ref{NormsOnInitialSurface} of $\mathcal{XS}^{2,\alpha}(\mathcal{M})$, we see that taking $\|v\|_{\mathcal{XS}^{2,\alpha}(\mathcal{M})}\leq 1$, we finally obtain:
\beq
\|\mathcal{F}_\Omega(v)-\mathcal{F}_\Omega(0)-\mathcal{L}_\Omega v:\mathcal{XS}^{0,\alpha}(\mathcal{M})\|\leq C\|v:{\mathcal{XS}^{2,\alpha}(\mathcal{M})}\|^2.
\eeq

For the core piece $\Sigma[\theta,\tau]$, the argument follows closely the one in \cite{Ka97}. Namely, using the uniform control on the geometry $\|A:C^3(\Sigma,g_\Sigma)\|\leq C$ and $\|\tau^2\vec{X}\cdot\vec{\nu}:C^0(\Sigma)\|\leq \tau$ with the expression for the quadratic term in Equation (\ref{FullVariation}), one obtains again that when $\|f: C^{2,\alpha}(\Sigma, g_\Sigma, e^{-\gamma s}+b_2)\|$ is small enough,
\begin{align*}
\|\mathcal{F}_\Sigma(f)-\mathcal{F}_\Sigma-\tau^2\mathcal{L}_\Sigma f&:C^{0,\alpha}(\Sigma,g_\Sigma, e^{-\gamma s})\|\\
\leq
C\|f&:C^{2,\alpha}(\Sigma,g_\Sigma, e^{-\gamma s}+b_2)\|^2.
\end{align*}

For the central disk and the top and bottom spherical caps the proofs are again the same, by uniform control of the geometry and (\ref{FullVariation}).
\end{proof}

\section{Fixed point argument: Existence of the self-shrinkers}

We consider for any fixed $0<\alpha'<\alpha<1$ the corresponding Banach space
\[
\mathcal{XS}^{2,\alpha'}:=\mathcal{XS}^{2,\alpha'}(\mathcal{M}[\tau,0]),
\]
from the family that we previously studied, and take the subsets
\[
\Xi=\{(\theta,u)\in [-\delta_\theta,\delta_\theta]\times\mathcal{XS}^{2,\alpha'}: |\theta|\leq\zeta\tau, \|u\|_{\mathcal{XS}^{2,\alpha}}\leq\zeta\tau\}.
\]
We state the following lemma (as in \cite{Ka97}), whose easy proof we omit.

\begin{lemma}
There is for $\theta\in [-\delta_\theta,\delta_\theta]$ a smooth family of diffeomorphisms
\[ \label{d_theta}
D_\theta: \mathcal{M}[\tau, 0]\to \mathcal{M}[\tau,\theta],\quad\textrm{with}
\]
\beq
\|f_1\circ D_\theta^{-1}\|_{\mathcal{XS}^{2,\alpha}}\leq C\|f_1\|_{\mathcal{XS}^{2,\alpha}},\quad\|f_2\circ D_\theta\|_{\mathcal{XS}^{2,\alpha}}\leq C\|f_2\|_{\mathcal{XS}^{2,\alpha}},
\eeq
for all $f_1\in C^{2,\alpha}(\mathcal{M}[\tau,0])$ and $f_2\in C^{2,\alpha}(\mathcal{M}[\tau,\theta])$.
\end{lemma}
The problem stated in Theorem \ref{LinearOnAll} is then continuous in $\tau$ and $\theta$ in the sense that, for fixed $E \in  \mathcal{XS}^{0, \alpha}(\mathcal{M}[\tau, 0])$, the pair 
\begin{equation} \notag
(v_{E \circ D_\theta^{-1}} \circ D_\theta, \theta_{E \circ D_{\theta}^{-1}} ) \in \mathcal{XS}^{2, \alpha}(\mathcal{M}[\tau, 0]) \times \mathbb{R}
\end{equation}
depends continuously on $\tau$ and $\theta$.

We define the map $\mathcal{J}:\Xi\to [-\delta_\theta,\delta_\theta]\times\mathcal{XS}^{2,\alpha'}$ as follows: Let $(\theta,u)\in\Xi$ and let $v:=u\circ D_\theta^{-1}-v_\mathcal{F}$, where $v_\mathcal{F}$ comes from an application of Corollary \ref{CorLinear}, and the function $\mathcal{F}=\mathcal{F}_{\mathcal{M}[\tau,\theta]}(0)$ as before is defined on $\mathcal{M}[\tau,\theta]$. We thus have
\[
\|v\|_{\mathcal{XS}^{2,\alpha}}\leq C(\zeta +1)\tau.
\]
Now, we use Proposition \ref{QuadraticImprovement} to get that $M_v$ is well-defined, and
\[
\|\mathcal{F}_{\mathcal{M}}(v)-\mathcal{F}_\mathcal{M}-\mathcal{L}_{\mathcal{M}}v\|\leq C(\zeta +1)^2\tau^2.
\]
Inserting therefore $E=\mathcal{F}_{\mathcal{M}}(v)-\mathcal{F}_\mathcal{M}-\mathcal{L}_{\mathcal{M}}v$ into Theorem \ref{LinearOnAll} gives a $v_E$ and $\theta_E$. We obtain, for some appropriate constant $C_0$ that:
\begin{align}
&\mathcal{F}_\mathcal{M}(v)=\mathcal{L}_{\mathcal{M}}\left(u\circ D_\theta^{-1}+v_E\right)-\Theta(\theta_{\mathcal{F}}+\theta_E),\\
&|\theta-\theta_\mathcal{F}-\theta_E|\leq C_0\big(\tau + (\zeta+1)^2\tau^2),\\
&\|v_E\|_{\mathcal{XS}^{2,\alpha}}\leq 2C_0(\zeta +1)^2\tau^2.
\end{align}

Then the definition of $\mathcal{J}$ is taken to be
\beq
\mathcal{J}(\theta,u)=\big(\theta-\theta_\mathcal{F}-\theta_E,-v_E\circ D_\theta\big).
\eeq

Thus, by assuming $\zeta$ large enough and $\tau>0$ small enough, we arrange that $\mathcal{J}(\Xi)\subseteq\Xi$. By the properties of our weighted spaces in Proposition \ref{Compactness} and $\alpha'<\alpha$, $\Xi$ is a compact subset of $[-\delta_\theta,\delta_\theta]\times\mathcal{XS}^{2,\alpha'}$ and is also convex.  The map $\mathcal{J}$ is continuous by Definition \ref{d_theta} and Proposition \ref{LinearOnAll}. Finally, by Schauder's fixed point theorem, we get existence of the desired fixed point $(\theta^*,u^*)\in\Xi$, and we see that the corresponding $\mathcal{M}[\tau,\theta^*]_{u^*}$ is an immersed self-shrinker. The rest of the proof of the main theorem now follows easily.

For example, embeddedness is assured by our setup: By construction, there is some fixed ball $B_{R_0}$ such that the end of every $\Sigma_g$ is graphical outside that ball, and hence embedded. Now, above one could pick $\zeta=2C_0$ independent of $\tau$, one concludes for all $\tau>0$ small enough in terms of this that $\|u^*\|_{\mathcal{XS}^{0,\alpha}}\leq C\tau^2$, and since also by construction the normal injectivity radius of a compact piece the initial surface, say $B_{R_0+1}(0)$, can be assumed bounded below as $\mathrm{inj}_\perp(\mathcal{M}[\tau,\theta]\cap B_{R_0+1}(0))\geq c\tau$, for some $c>0$, it follows that (for possibly even smaller $\tau>0$) the constructed surfaces $\Sigma_g$ are embedded. The Hausdorff convergence statement (v) in Theorem \ref{Thm_main} also follows immediately from the definitions of the norms.

It also follows easily that each surface $\Sigma_g$ is geodesically complete. Namely, a curve that leaves every compact set must have infinite length, as follows by projecting it onto the plane and estimating the arc length from below, again since the ends are graphical outside some ball.

\newpage

% -- Beware, beyond lurks Appendix A! --
%----------------------------------------
\section*{Appendix A: The building blocks of the initial surfaces}

The proof of Proposition \ref{shrinker_caps} follows from the following lemma.
\begin{lemma}\label{Caps}
There exist $\delta,\varepsilon_0,\varepsilon_1>0$ such that there is a smooth map $h\mapsto \rho_h$ for $h\in (2-\delta,2+\delta)$ such that 
\begin{itemize}
\item For each $h$, the function $\rho_h$ is a generates a curve contained in the set $\{(x, 0 , z): x, z \geq 0\}$.
\item $\rho_0 (\varphi) \equiv 2$, with $\varphi = \arctan (x /z)$ ,  and  $\rho_h(0)=2+h, \rho_h'(0)=0$.
\item The following are orientation-reversing diffeomorphisms:
\begin{itemize}
\item[(1)] $h\mapsto \rho_h(0):(2-\delta,2+\delta)\to(2-\varepsilon_0,2+\varepsilon_0)$,
\item[(2)] $h\mapsto \rho'_h(\pi/2):(2-\delta,2+\delta)\to(-\varepsilon_1,\varepsilon_1)$.
\end{itemize}
\item The graph  $(\rho_h(\varphi), \varphi)$, $\varphi\in(0,\pi/2)$, in the $x z$-plane gives by revolution a self-shrinker.
\end{itemize}
\end{lemma}

\begin{proof}[Proof of Lemma \ref{Caps}]
A curve $(\rho,\varphi), \varphi = \arctan (x/ z)$ in the $x z$-plane generated by a function $\rho(\varphi)$  that generates a smooth solution to the self-shrinker equation (\ref{SSEq}) satisfies:
\beq
\rho''(\varphi)=\frac{1}{\rho}\left\{\rho^2+2(\rho')^2+\Big[1-\frac{\rho^2}{2}-\frac{\rho'}{\rho\tan\varphi}\Big]\big(\rho^2+(\rho')^2\big)\right\}.
\eeq

The Taylor-expansion in the Banach space of $C^2$ functions of the solution in the $h$-parameter is
\[
\rho_h(\varphi)=2+(h-2)w_1(\varphi)+\frac{(h-2)^2}{2}w_2(\varphi)+O\big((h-2)^3\big),
\]
where $w_i$ are smooth functions. The $w_i$ satisfy the conditions $w_1(0)=1$ and $w_1'(0)=0$ (and $w_2(0)=w_2'(0)=0$ and similarly for higher corrections), and as is easily computed $w_1$ satisfies the linear equation
\beq\label{PolLin}
w_1''+\frac{1}{\tan\varphi}w_1'+4w_1=0,
\eeq
while $w_2$ satisfies a linear equation where the $w_1$ enters into the coefficients.

The claims (1) and (2) follow from the following two properties
\begin{align}
&w_1(\fracsm{\pi}{2})< 0,\\
&w_1'(\fracsm{\pi}{2})< 0,
\end{align}
for the solution to (\ref{PolLin}) having $w_1(0)=1$ and $w_1'(0)=0$.

In fact since if we subsitute $x=\cos(\varphi)$ in the equation (\ref{PolLin}) to obtain Legendre's differential equation, the explicit general solution to this initial value problem is of course well-understood, namely
\[
w_1(\varphi)=C_1 P_l(\cos \varphi)+C_2Q_l(\cos\varphi),
\]
where $P_l$ and $Q_l$ are respectively the Legendre functions of the first and second kind, and $l=(\sqrt{17}-1)/2$ is the positive solution to $l(l+1)=4$. Here we see $C_2=0$, since $Q_l(\cos\varphi)$ has a pole at $\varphi=0$, and  $C_1=1$ since $P_l(1)=1$. Thus the properties are easily verified and the lemma follows.
\end{proof}

\begin{proof}[Proof of Proposition \ref{shrinker_caps}]
Given $\rho_h$ constructed above, set $\theta = \tan \{\rho_h'(\pi/2)\}$.  We then take $\mathcal{K}_\theta$ to be the surface immersed by the map $\kappa_\theta$ given by
\begin{equation} \notag
\kappa_\theta (s, z) = r (\varphi(s)) (\cos z, \sin z, 0) + (0, 0, z(\varphi(s)))
\end{equation}
where $r(\varphi) =  \rho_h (\varphi) \sin(\varphi)$, $z (\varphi) = \rho_h \cos (\varphi)$, and the map $s \mapsto \varphi (s)$ satisfying
\begin{equation} \notag
s_\varphi =\frac{ \sqrt{ r_\varphi^2 + x_{3,\varphi}^2}}{r (\varphi)},\quad s(\pi/2) = 0.
\end{equation}
That (0)-(3) are satisfied by the family $\mathcal{K}_\theta$ are clear by construction. Likewise, once it is checked that $s(\varphi)$ is a conformal parameter,   (4)i - iv are easy to verify.  
\end{proof}

\section*{Appendix B: Variation formulae}
Let $\vec{X}:M\to\Reals^3$ be a $C^2$-immersion of a surface. We denote by $\vec{X}_u:M\to\Reals^3$ the surface $\vec{X}_u=\vec{X}+u\vec{\nu}$. Then denoting by $H_u$ and $\vec{\nu}_u$ etc. the quantities for $\vec{X}_u$, we get (see \cite{Ka97} and \cite{Ng1})
\begin{align}
&\vec{\nu}_u=\vec{\nu}-\nabla u+\vec{Q}^\nu_u,\\
&H_u=H-(\Delta u+|A|^2u)+Q_u,\label{MinimalOnMinimal}\\
&H_u -\fracsm{1}{2}\tau^2\vec{X}_u\cdot\vec{\nu}_u=
H-\fracsm{1}{2}\tau^2 \vec{X}\cdot\vec{\nu}\label{FullVariation}\\
&\phantom{H_u -\fracsm{1}{2}\tau^2\vec{X}_u\cdot\vec{\nu}_u=}
-\big[\Delta u+|A|^2u -\fracsm{1}{2}\tau^2 (\vec{X}\cdot\nabla u-u)\big]\notag\\
&\phantom{H_u -\fracsm{1}{2}\tau^2\vec{X}_u\cdot\vec{\nu}_u=}
+ Q_u + \fracsm{1}{2}\tau^2 \vec{X}\cdot\vec{Q}_u\notag,
\end{align}
where the quantities $Q_u$ and $\vec{Q}_u$ are quadratic.

\section*{Appendix C: Stability operators}
Let $M^2\hookrightarrow N:=(\Reals^{3},h=e^{2\omega}h_0)$ be an immersion into a conformally changed Euclidean space, where $\omega:N\to\Reals$. Here $h_0$ will denote the standard metric $h_0=\delta_{ij}$. Denote by $g_0$ and $g$ the metrics induced on $M^2$ from respectively $h_0$ and $g$ by the immersion.

Here we have the conventions:
\begin{align*}
&\Delta f=\diverg(\grad f) =\trace(\nabla_i\partial_j f),\\
&R(X,Y)Z=\nabla_Y\nabla_XZ-\nabla_X\nabla_YZ+\nabla_{[X,Y]}Z\\
&\Ric(X,Y)=\trace(Z\mapsto R(X,Z)Y),
\end{align*}
so that the Ricci curvature of the standard round sphere is positive.

Then we have the following Lemma.
\begin{lemma}
Assume $M^2$ is an oriented minimal surface in $N$. Then the stability operator of $M^2$ is the operator on functions on $M^2$ given by:
\begin{align*}
L_g&=\Delta_g+|A_h|^2_g+\Ric^h(\vec{\nu},\vec{\nu})\\ &=e^{-2\omega}\Big[\Delta_{g_0}+e^{-2\omega}|A_{h}|_{g_0}^2-\Hess_{h_0}(\vec{\nu}_0,\vec{\nu}_0)\omega+(\vec{\nu}_0.\omega)^2-\Delta_{h_0}\omega-\|\nabla_{h_0}\omega\|^2_{h_0}\Big],
\end{align*}
where $\vec{\nu}_0$ is the unit normal vector w.r.t the metric $h_0$.
\end{lemma}

From this formula, we get the stability operators:
\begin{proposition}\label{PropStab}
\begin{itemize}
\item[]
\item[(i)] The stability operator of the sphere $\mathbb{S}^2$ of radius $2$ in $\Reals^3$ as a minimal surface in the metric $g=e^\frac{-|x|^2}{4}\delta$ is:
\beq
L=e\Big(\Delta_{\mathbb{S}^2_{2}}+1\Big).
\eeq
In particular $\ker(L)=\{0\}$ on $\mathbb{S}^2_2$, as well as on the hemispheres of radius 2 with Dirichlet boundary conditions.

\item[(ii)] The stability operator on a flat plane through the origin is
\beq\label{StabPlane}
L=e^{\frac{|x|^2}{4}}\Big(\Delta_{\Reals^2}-\frac{|x|^2}{16}+1\Big),
\eeq
where $\Delta_{\Reals^2}$ is the usual flat Laplacian in $(\Reals^2,\delta_{ij})$. In particular, on both the disk of radius $\sqrt{2}$, and of radius 2, $\ker(L)=\{0\}$ when we impose Dirichlet boundary conditions.

\end{itemize}
\end{proposition}

\begin{proof}
Recall that by definition $A(X,Y)=\bar{\nabla}_{\bar{X}}\bar{Y}-\nabla_XY$, where $\bar{\cdot}$ means a smooth extension to a neighborhood in $N$, and
\[
|A|^2=g^{ij}g^{kl}a_{ik}a_{jl}=e^{-4\omega}g_0^{ij}g_0^{kl}a_{ik}a_{jl}.
\]
We also recall the conformal changes of the Levi-Civitas:
\begin{align}
&\nabla^h_{\bar{X}}\bar{Y}=\nabla^{h_0}_{\bar{X}}\bar{Y}+(\bar{X}.\omega)Y+(\bar{Y}.\omega)\bar{X}-h_0(\bar{X},\bar{Y})\nabla^{h_0}\omega\\
&\nabla^g_XY=\nabla^{g_0}_XY+(X.\omega)Y+(Y.\omega)X-g_0(X,Y)\nabla^{g_0}\omega.
\end{align}
This gives that
\beq\label{AChange}
A_h(X,Y)=A_{h_0}(X,Y)-g_0(X,Y)\big\{\nabla^{h_0}\omega-\nabla^{g_0}\omega\big\}
\eeq

The Ricci curvature changes in dimension $n=3$ when $h_0=\delta$ according to:
\begin{align*}
\Ric^h=&\Ric^{h_0}-(n-2)\Big[\nabla^{h_0}\dRham\omega-\dRham\omega\otimes\dRham\omega\Big]+\Big[-\Delta^{h_0}\omega -(n-2)\|\nabla^{h_0}\omega\|_{h_0}^2\Big]h_0,\\
=&-\Hess^{h_0}\omega+\dRham \omega\otimes\dRham\omega-\Big[\Delta^{h_0}\omega +\|\nabla^{h_0}\omega\|_{h_0}^2\Big]h_0.
\end{align*}

Here we have used that for $h_0=\delta$ we have
\[
\nabla^{h_0}\dRham\omega =\Hess\omega,
\]
and that $\vec{\nu}=e^{-\omega}\nu_0$ is the new unit normal.

Recall also that in 2 dimensions the Laplacian is conformally covariant:
\beq\label{ConfCovar}
\Delta_{g}=e^{-2\omega}\Delta_{g_0}.
\eeq
Using these formulae, the Lemma follows.

To prove the proposition, we need $\omega = -\frac{|x|^2}{8}$. Thus we have
\begin{align*}
&\nabla^{h_0}\omega=-\frac{1}{4}x,\quad\Hess\omega(\partial_i,\partial_j)=-\frac{1}{4}\delta_{ij},\\
&\Delta\omega=-\frac{3}{4},\quad\|\nabla^{\Reals^3}\omega\|_{\Reals^3}^2=\frac{1}{16}|x|^2.
\end{align*}

And we get on $\Reals^3$ that
\begin{align*}
(\Ric_h)_{ij}=&-\Hess_{h_0}\omega(\partial_i,\partial_j)+(\partial_i\omega)(\partial_j\omega)-\Big[\Delta^{\Reals^3}\omega+\|\nabla^{\Reals^3}\omega\|^2\Big]\delta_{ij}\\
=&\frac{1}{4}\delta_{ij}+\frac{1}{16}x_ix_j+\frac{3}{4}\delta_{ij}-\frac{|x|^2}{16}\delta_{ij}=\delta_{ij}+\frac{1}{16}x_ix_j-\frac{|x|^2}{16}\delta_{ij}.
\end{align*}

Thus we get
\beq
\Ric^{h}(\vec{\nu},\vec{\nu})=e^{-2\omega}\left[1+\frac{|x\cdot\nu|^2}{16}-\frac{|x|^2}{16}\right],
\eeq

so that on the round sphere of radius 2,
\beq
\Ric^{h}(\vec{\nu},\vec{\nu})=e^{-2\omega}.
\eeq

Now, we pull back the induced metric $g_{2}$ on $\mathbb{S}^2_{2}$ of radius $2$ by the map $\Phi(x)=2x$ taking $\mathbb{S}^2_{1}\to \mathbb{S}^2_{2}$ to get the isometry $(\mathbb{S}^2_{1},\Phi^*g_{2})\simeq(\mathbb{S}^2_{2},g_{2})$. Then note that for $X,Y\in T\mathbb{S}^2_{1}$ we have $\Phi^*g_{2}(X,Y)=g_{\Reals^3}(d\Phi(X),d\Phi( Y))=4g_{\Reals^3}(X,Y)=4g_{1}$. Thus by the covariance in Equation (\ref{ConfCovar}), the spectrum of the operator $L$ is the same as that of $\Delta_{\mathbb{S}^2_1}+4$ on the sphere of radius 1.

Now, since the eigenvalues of $\Delta$ on the unit sphere $\mathbb{S}^2=\mathbb{S}^2_1$ are
\[
\lambda_k=-k(k+1),
\]
we see that $\Delta_{\mathbb{S}^2_1}+4$ is invertible on the sphere. The eigenvalues for the Dirichlet problem for $\Delta$ on the hemispheres are the same, but with smaller multiplicity (and in particular $0$ is not an eigenvalue). Thus $\Delta_{\mathbb{S}^2_1}+4$ is also invertible there.

Considering the plane $\{z= 0\}$, one gets similarly $A=0$, and
\beq
\Ric^{h}(\vec{\nu},\vec{\nu})=e^{-2\omega}\Big(1-\frac{|x|^2}{16}\Big).
\eeq

Recall that for the Dirichlet problem for the harmonic oscillator on the unit disk, we have that $\lambda_k=-k^2$, where $k=1,2,3,\ldots$ are the integers. Thus $\lambda_k=-\frac{k^2}{2}$ on the disk of radius $\sqrt{2}$, while $\lambda_k=-\frac{k^2}{4}$ on the disk of radius 2. Thus in either case the corresponding stability operator $L$ is invertible.
\end{proof}

\section*{Appendix D: Laplacians on flat cylinders}

We here recall a simple analytical result on $(\Omega,g_0)$ the flat cylinder $\Omega=H^{+}_{\leq l}/G$ equipped with the standard metric $g_0=ds^2 +dz^2$, where $G$ is the group generated by $(s,z) \to (s,z+2\pi)$, and $l \in (10, \infty)$ is called the length of the cylinder. We have $\partial \Omega=\partial_0 \cup \partial _l$ where $\partial_0$ and $\partial_l$ are the boundary circles $\{s=0\}$ and $\{s=l\}$ respectively. 

Let $\mathcal{L}$ on the flat cylinder $(\Omega, g_0)$ be given by
 \beq
	\mathcal{L} v = \Delta_{\chi} v + \mathbf{A} \cdot \nabla v+ B\cdot v,
 \eeq
where $\chi$ is a $C^2$ Riemannian metric, $\mathbf{A}\in C^1(\Omega, \Reals^2)$ is a vector field, and $B\in C^1(\Omega)$. We define 
	\begin{equation*}
	N(\mathcal{L}) := 
	\left\{ \| \chi -g_0:C^2(\Omega,g_0) \| + \| \mathbf{A} :C^1(\Omega,g_0 )\|  + \| B :C^1(\Omega,g_0 )\| \right\} 
	\end{equation*}

\begin{proposition} \label{laplacian_on_cylinders}
Given $\gamma \in (0,1)$ and $\ep>0$, if $N(\mathcal{L})$ is small enough in terms of  $\alpha$,  $\gamma$ and $\ep$, then
there is a bounded linear map 
	\[
	\underline{\mathcal{R}}: C^{2,\alpha}(\partial_0,g_0) \times C^{0, \alpha}(\Omega, g_{0}, e^{-\gamma s}) \to C^{2, \alpha}(\Omega, g_{0}, e^{-\gamma s})
	\]
such that for $(f,E)$ in the domain of $\underline{\mathcal R}$ and $v = \underline{\mathcal R}(f,E)$, the following properties are true, where the constants $C$ depend only on $\alpha$ and $\gamma$:
	\begin{enumerate}
	\item   $\mathcal{L} v = E$ on $\Omega$.
	\item  $v= f -\textrm{\em avg}_{\partial_0} f +B(f,E)$ on $\partial_0$, where $B(f,E)$ is a constant on $\partial_0$ and $\textrm{\em avg}_{\partial_0} f$ denotes the average of $f$ over $\partial_0$. 
	\item  $v\equiv 0$ on $\partial_l$.
	\item $
	\| v :C^{2, \alpha}(\Omega, g_{0}, e^{-\gamma s}) \| $\\
	$\leq C \| f-\textrm{\em avg}_{\partial_0} f:C^{2, \alpha}(\partial_0, g_{0})\|+ C \| E: C^{0, \alpha}(\Omega, g_{0}, e^{-\gamma s})\|.
	$
	\item $|B(f,E)|\leq \ep\| f-\text{\em avg}_{\partial_0} f: C^{2, \alpha}(\partial_0, g_{0})\|+ C \| E: C^{0, \alpha}(\Omega, g_{0}, e^{-\gamma s})\|.
	$
	\item If  $E$ vanishes, then 
	\[
	\| v : C^0(\Omega) \| \leq 2 \| v:C^0(\partial_0)\|.
	\]
	\end{enumerate}
Moreover, the function $v$ depends continuously on  $\mathcal{L}$. 
\end{proposition}
\begin{proof}
For a metric $\chi = (\chi_{ij})$ in local coordinates, we can write the Laplace operator $\Delta_\chi$ as
\begin{equation} \label{chi_laplacian}
\Delta_\chi = \chi^{ij} \partial_{ij} + \chi^{ij}_{,i} \partial_j + \chi^{k j} \Gamma_{ik}^i \partial_j
\end{equation}
where $\chi^{-1} = (\chi^{ij})$ is the inverse matrix for $(\chi_{ij})$, and where 
\begin{equation}
\Gamma_{ij}^k = \frac{1}{2} \chi^{lk}\left(\chi_{li, j} + \chi_{jl, i} - \chi_{ij, l} \right)
\end{equation}
 are the Christoffel symbols for the Riemannian connection for $\chi$. By (\ref{chi_laplacian}) we have that
 \begin{equation} \notag
 \|\mathcal{L}- \Delta_{g_0} \| \leq C \mathcal{N}(\mathcal{L})
 \end{equation}
 where $\| \mathcal{L} - \Delta_{g_0} \|$ denotes the operator norm of $\mathcal{L} - \Delta_{g_0}$ as a map from $C^{2, \alpha}(\Omega, g_{0}, e^{-\gamma s})$ to $C^{0, \alpha}(\Omega, g_{0}, e^{-\gamma s})$. Thus by taking $\mathcal{N} (\mathcal{L})$ sufficiently small we can arrange so that 
 \begin{equation}
 \| \mathcal{L} - \Delta_{g_0}\| < \delta
 \end{equation}
 for any $\delta > 0$. Despite the presence of small $L^2$ eigenvalues for the flat laplacian $\Delta_{g_0}$ on a long cylinder we can still define a uniformly bounded inverse as follows: Given a function $E \in C^{0, \alpha}(\Omega, g_{0}, e^{-\gamma s})$, write $E(s, z) = E_0(s, z)+ e_0(z)$, with $e_0 (z) = \frac{1}{2 \pi }\int_{ \sigma = z} E(s, \sigma) ds$  the radial average of  $E$. Then for any function $f \in C^{2, \alpha}(\partial_0)$ we can solve
 \begin{eqnarray} \notag
 \Delta_{g_0} U_0  & = & E_0  \\ \notag
 U_0 & = & f - \text{avg}_{\partial_0} f \text{ on } \partial_0  \\ \notag
 U_0 & = & 0 \text{ on } \partial_l 
 \end{eqnarray}
 with 
 \begin{equation} \notag
 \|U_0:  C^{2, \alpha}(\Omega, g_{0}, e^{-\gamma s}) \| \leq C \|E_0 :  C^{0, \alpha}(\Omega, g_{0}, e^{-\gamma s}) \| +C \| f - \text{avg}_{\partial_0} f: C^{2, \alpha}(\partial_0)\|.
 \end{equation}
 The radial part which projects onto the small eigenvalues is then directly integrated by setting  $u_0(z) = \int_z^l \int_s^l e_0(t) dt ds$. We then have
 \begin{eqnarray} \notag
 \mathcal{L} (U_0 + u_0) & = & E + (\mathcal{L} - \Delta_{g_0}) (U_0 + u_0) := E + E_1 \\ \notag
 U_0 + u_0 & = & c_0 \text{ on } \partial_0 \\ \notag
 U_0 + u_0 & = & 0 \text{ on } \partial_l \notag
 \end{eqnarray}
 where $E_1$ is defined by the  equality above and satisfies 
 \begin{eqnarray} \notag
 \|E_1:  C^{0, \alpha}(\Omega, g_{0}, e^{-\gamma s}) \|  &\leq &  \delta \|U_0 + u_0:  C^{2, \alpha}(\Omega, g_{0}, e^{-\gamma s}) \| \\ \notag
 & \leq & \delta (C_0 + 1) \|E:  C^{0, \alpha}(\Omega, g_{0}, e^{-\gamma s}) \| \notag
 \end{eqnarray}
 where $C_0$ denotes the operator norm of $\Delta_{g_0}^{-1}$ in the space of $L^2$ functions with zero radial average.  The process is then iterated to obtain a sequence $\{(U_k, u_k) \}_{k = 1}^{\infty}$ satisfying
 \begin{eqnarray} \notag
 \Delta_{g_0} U_k & = & (\mathcal{L} - \Delta_{g_0})(U_{k - 1 } + u_{k - 1}) - e_k, \\ \notag
 U_k  & = & 0 \text{ on } \partial \Omega, \notag
 \end{eqnarray}
with
 \begin{equation} \notag
 e_k(z) = \int_{\sigma = z}  (\mathcal{L} - \Delta_{g_0})(U_{k - 1 } + u_{k + 1})(s, \sigma) ds
 \end{equation}
 and
 \begin{equation}
 u_k =   \int_z^l \int_s^l e_k(t) dt ds
 \end{equation}
 Choosing $\delta$ so that $\delta C_0 = \epsilon' < 1$, we than have that
 \begin{equation}
 \|U_k:  C^{2, \alpha}(\Omega, g_{0}, e^{-\gamma s})  \|, \| u_k:  C^{2, \alpha}(\Omega, g_{0}, e^{-\gamma s})\| < \epsilon'^k \|E:  C^{0, \alpha}(\Omega, g_{0}, e^{-\gamma s}) \|
 \end{equation}
 The alternating partial sums $v_k = \Sigma_{i = 0}^k (-1)^i (U_i + u_i)$ then converge to a function $v$ satisfying $(1) - (6)$ above. The continuous dependence on  $\mathcal{L}$ follows directly by construction.
 
 \end{proof}
\begin{remark}
The reader will note that a similar proposition was first recorded in  \cite{Ka97} and \cite{Ka95}, and is a fundamental part of the linear theory in both these articles. The proposition recorded here differs from the previous versions in that we allow a much broader class of perturbations at the expense of a uniqueness claim.
\end{remark}
\bibliographystyle{amsalpha}

\end{document}